\documentclass[letterpaper,11pt]{article}
\usepackage{tabularx} 
\usepackage{graphicx} 
\usepackage[margin=1in,letterpaper]{geometry} 
\usepackage{cite} 
\usepackage[final]{hyperref} 
\hypersetup{
	colorlinks=true,       
	linkcolor=blue,        
	citecolor=blue,        
	filecolor=magenta,     
	urlcolor=blue         
}
\usepackage{blindtext}
\usepackage[all,cmtip]{xy}
\usepackage{amsthm}
\usepackage{libertine}
\usepackage{hyperref}
\usepackage[utf8]{inputenc}
\usepackage{amsmath,amssymb}
\usepackage[all,cmtip]{xy}
\usepackage{mathrsfs}
\usepackage{mathtools}
\usepackage{multicol}
\usepackage{tikz-cd}
\usepackage[shortlabels]{enumitem}
\usepackage{siunitx}
\usepackage{tikz}
\usetikzlibrary{positioning,arrows}

\usepackage{cancel}
\usepackage{graphicx}
\usepackage{dsfont}
\usepackage{tikz}
\usepackage{tikz-cd}
\usepackage{pgfplots}
\usepackage{listings}
\usepackage{pdfpages}
\newtheorem{theorem}{Theorem}[section]
\newtheorem{conjecture}{Conjecture}
\newtheorem{theorem*}{Theorem}
\newtheorem{lemma}[theorem]{Lemma}
\newtheorem{corollary}[theorem]{Corollary}
\newtheorem{proposition}[theorem]{Proposition}
\newtheorem{remark}[theorem]{Remark}

\newtheorem{thmx}{Theorem}

\DeclareMathOperator{\Alg}{alg}

\DeclareMathOperator{\ST}{ST}
\DeclareMathOperator{\Gl}{GL}

\DeclareMathOperator{\Amp}{Amp}

\DeclareMathOperator{\NS}{NS}
\DeclareMathOperator{\Hom}{Hom}
\newcommand*{\SheafHome}{\mathcal{H}\kern -.5pt om}
\newcommand*{\SheafHom}{\mathcal{RH}\kern -.5pt om}
\DeclareMathOperator{\Pic}{Pic}
\DeclareMathOperator{\Coh}{Coh}
\DeclareMathOperator{\Hilb}{Hilb}
\DeclareMathOperator{\Fix}{Fix}

\DeclareMathOperator{\Ann}{Ann}

\DeclareMathOperator{\Mov}{Mov}
\DeclareMathOperator{\Br}{Br}
\DeclareMathOperator{\Reldim}{reldim}

\DeclareMathOperator{\Ext}{Ext}
\DeclareMathOperator{\Deg}{deg}
\DeclareMathOperator{\Id}{id}

\DeclareMathOperator{\Bl}{Bl}

\DeclareMathOperator{\Stab}{Stab}

\DeclareMathOperator{\OGr}{OGr}
\DeclareMathOperator{\LGr}{LGr}
\DeclareMathOperator{\R}{R\Gamma}

\DeclareMathOperator{\Dim}{dim}

\DeclareMathOperator{\RHom}{RHom}
\DeclareMathOperator{\Rank}{rank}
\newcommand*{\shom}{\mathscr{H}\kern -.5pt om}
\newcommand*{\sext}{\mathscr{E}\kern -.5pt xt}
\makeatletter
\newcommand{\subjclass}[2][2020]{%
  \let\@oldtitle\@title%
  \gdef\@title{\@oldtitle\footnotetext{#1 \emph{Mathematics subject classification.} #2}}%
}
\newcommand{\keywords}[1]{%
  \let\@@oldtitle\@title%
  \gdef\@title{\@@oldtitle\footnotetext{\emph{Key words and phrases.} #1.}}%
}
\makeatother
\begin{document}
\title{Fano visitor problem for K3 surfaces}
\author{Anibal Aravena}
\date{}
\maketitle
\begin{abstract}
We study the Fano visitor problem for a given K3 surface $X$ with Picard number 1 and genus $g\geq 4$. This problem  consists in constructing a smooth Fano variety $Y$ such that its bounded derived category $D^b(Y)$ contains $D^b(X)$ as an admissible category. If such variety exists, then $X$ is called a Fano visitor. We prove that if  $g\not\equiv 3\mod 4$, then $X$ is a Fano visitor. In the case  $g\equiv 3\mod 4$, we construct a  smooth weak Fano variety. Our proof is based on a detailed study of a sequence of birational maps involving Lagrangian subvarieties of several $2g$-dimensional moduli spaces of stable objects in  $D^b(X)$ constructed by Flapan, Macr\`i, O'Grady and Sacc\`a in \cite{Flapan_2021}. In this paper, We prove that this sequence of birational maps corresponds to running the minimal model program (MMP) for the blow-up $M_1=\Bl_X\mathbb{P}^g$, where  $X\hookrightarrow \mathbb{P}^g$ is the embedding given by the complete linear system associated with the ample generator of $\Pic X$.  We  fully describe  the stable base locus decomposition $M_1$. From this, we obtain $Y$ as a small modification of $M_1$, and by results of Bondal-Orlov  \cite{bondal}, and Belmans, Fu and Raedschelders \cite{belmans}, we stablish the embedding $D^b(X)\hookrightarrow D^b(Y)$. Moreover, we construct a full semiorthogonal decomposition for $D^b(Y)$. In particular, we deduce that several moduli space of sheaves on $X$ are also Fano visitors. This result provides new families  of Fano visitors,  which are hyperkh\"aler varieties  of arbitrarily high dimension.
\end{abstract}

\section{Introduction}
Let $X$ be a smooth projective variety and $D^b(X)$ its bounded derived category of coherent sheaves. We say that $X$  is a Fano visitor if  there exists a smooth Fano variety $Y$ and a fully faithful embedding
$$ D^b(X)\hookrightarrow D^b(Y).$$
If $Y$ is a weak Fano variety, i.e., $-K_Y$ is big and nef, then we say that $X$ is a weak Fano visitor. In 2011, Bondal raised the question if every smooth projective variety  is a Fano visitor (see, e.g., \cite[Pag~185]{BERNARDARA2016181}). Examples of Fano visitors include  curves and their symmetric powers and Jacobians, Enriques surfaces, and complete intersections (see \cite{nara1,nara2,leemoon,tevelev2023bgmn,Kuznetsov_2019,KIEM2017649}).

Let  $X$ be  a K3 surface of Picard number 1.  Then $\Pic X$ is cyclic and generated by an ample line bundle $\Lambda$. We define  the genus of  $X$  as  the integer $g$ such that $h^2=2g-2$, where $h=c_1(\Lambda)$ is the first Chern class of $\Lambda$. The main result of this paper is the following.
\begin{thmx}\label{A}
   Let $X$ be a K3 surface of Picard number 1 and  genus $g$. If  $g\not\equiv 3 \mod 4$, then $X$ is a Fano visitor. If $g\equiv 3\mod 4$, then $X$ is a weak Fano visitor.
\end{thmx}
Our proof  for  Theorem \ref{A} is based on the existence of a special sequence of birational maps involving several smooth compact complex manifolds depending  on $X$ and the ample line bundle $\Lambda$. This sequence always exists in the case when $X$ is a curve, and thus, our proof can be viewed as an extension of this argument to K3 surfaces.  To be more precise, let us  temporarily consider an arbitrary  smooth projective variety  $X$ with an ample line  bundle  $\Lambda$  such that $\omega_X\otimes \Lambda$ is very ample. Define $$M_0=\mathbb{P}(H^0(X,\omega_X\otimes \Lambda)^*),\qquad M_1=\Bl_XM_0,$$ where $X $ is embedded into $M_0$ by the  linear system $|\omega_X\otimes \Lambda|$. Suppose that there is a sequence  of smooth complex projectives varieties $M_i$, $1\leq i\leq \nu$, and  birational maps 
 \begin{equation}\label{flips}
     M_1\dasharrow  M_2\dasharrow M_3\dasharrow \ldots  \dasharrow M_\nu 
 \end{equation}
 satisfying the following conditions:
 \begin{enumerate}[i)]
     \item  The birational map $M_{i}\dasharrow M_{i+1}$ is small. Moreover, it decomposes as a sequence of standard flips. Recall that a standard flip $M\dasharrow M'$ is a birational map that fits into a diagram:
     $$ \xymatrix{
   &   &  \overset{\sim}{M} \ar[dl] \ar[dr]  &  & \\
   P \ar[drr] \ar@{^{(}->}[r]  &  M \ar@{-->}[rr] & &  M' & P' \ar@{_{(}->}[l]  \ar[dll]\\
    &  &  Z & &  }
$$
where $P,P'$ are projective bundles over $Z$, and $\overset{\sim}{M}$ is the blowup  of $M$ and $M'$  along $P$ and $P'$ respectively.   
\item There is some $j\leq \nu$ such that $M_j$ is Fano.
\item For  $i\leq j$, and for each pair of projective bundles $P\to Z$ and $P'\to Z$ that are flipped under  $M_i\dasharrow M_{i+1}$, we have the inequality 
$$ \Reldim(P\to Z)\leq \Reldim(P'\to Z).$$
 \end{enumerate}
 Then, the result of Bondal and Orlov \cite[Theorem 3.6]{bondal} provides a sequence of fully faithful embeddings 
 $$ D^b(X)\hookrightarrow D^b(M_1)\hookrightarrow D^b(M_2)\hookrightarrow \ldots \hookrightarrow D(M_j),$$
and so $X$ is a Fano visitor.  

When $X$ is a curve, such sequence was constructed by Thaddeus in \cite{Thaddeus1994}. Here  each  $M_i$ is a moduli space of pairs $(E,s)$, consisting of a non-zero section $s\in H^0(X,E)$ and a vector bundle  $E$ on $X$ of rank two and determinant $\Lambda$, subjected to some stability condition. Variation of this stability condition  induces  a small birational $M_i\dasharrow M_{i+1}$ which is a single standard flip, and (at least initially) the  projective bundle  is replaced  by another  of larger rank until a Fano model $M_j$ is reached (for some $j\leq \nu$). Moreover, by considering $\Lambda\gg 0 $, the same analysis shows that all symmetric powers of $X$ are also Fano visitors (see \cite[Proposition~3.18]{tevelev2023bgmn}).

When $X$ is a K3 surface with Picard number 1 and $\Lambda$ is the ample generator of $\Pic X$, a similar sequence of standard flips can be obtained from \cite[Section~3]{Flapan_2021} by studying   moduli spaces of (Bridgeland) stable complexes in $D^b(X)$ with Mukai vector
  $$ v\coloneqq (0,h,1-g) \in H^*(X,\mathbb{Z}),$$
  which are fixed by the contravariant functor
\begin{equation*} \Psi:D^b(X)^{op}\to D^b(X),\qquad E\mapsto \SheafHom(E,\Lambda^*[1]).
\end{equation*}
More precisely, we start by considering the moduli space $\mathcal{M}=\mathcal{M}_h(v)$ of Gieseker stable sheaves on $X$ with Mukai vector $v$. It has a Lagrangian fibration $\pi:\mathcal{M}\to \mathbb{P}^g$, called  the Beaville-Mukai system,  which corresponds to  the compactified relative Jacobian of the complete linear system $|\Lambda|$  \cite{mukaisystem,beauvillesystem}.  By the work of Bridgeland \cite{stability} and Bayer-Macr\`i  \cite{MR3194493,mmp},  we know that $\mathcal{M}$  coincides with the moduli space of Bridgeland stable complexes for some Bridgeland stability condition, and  by deforming  this stability condition, we obtain different  hyperk\"ahler (HK) varieties that are birational to  $\mathcal{M}$. Moreover, in this case, there is a  one parameter family of stability conditions which induces a diagram of flops (see \cite[Lemma~3.19]{Flapan_2021}).
\begin{equation}\label{flops}
    \resizebox{0.91\hsize}{!}{%
   $\xymatrixrowsep{1.4pc}\xymatrixcolsep{1.6pc} \xymatrix{
 \mathcal{M}=\mathcal{M}_{-1} \ar[d]_{\pi}  \ar@{-->}[rr] & &
\mathcal{M}_0  \ar@{-->}[rr] &   & \mathcal{M}_1 \ar@{-->}[rr]  &      &  \ldots   \ar@{-->}[rr] &   &  \mathcal{M}_\eta  \\
\mathbb{P}^g & & & & & & & &   }$%
}
\end{equation}
 The contravariant functor  $\Psi$ then acts on each space moduli space $\mathcal{M}_i$, inducing an antisymplectic involution $\tau$ on $\mathcal{M}_i$. This sequence of flops induces  a sequence of birational maps
\begin{equation}\label{flip2}
    M_0\dasharrow M_1\dasharrow M_2\dasharrow \ldots \dasharrow M_\nu,
\end{equation}
(we ignore $\mathcal{M}_{-1}$), where each $M_i$ is a connected component of $\Fix(\tau,\mathcal{M}_i)$. Moreover, we have that (see Proposition \ref{blowup}) $$M_0\simeq \mathbb{P}(H^0(X,\Lambda)^*),\qquad M_1=\Bl_XM_0.$$ In this paper, we show that $M_i\dasharrow M_{i+1}$ satisfies condition $i)$ \footnote{where $P,P'$ are possibly twisted projective bundles.} for $i\geq 1$ (see Proposition \ref{smallmod}). Here, the condition that the birational map decomposes into one or more standard flips is necessary, as in general $M_i\dasharrow M_{i+1}$ is not a single standard flip.  If $g\not\equiv  3 \mod 4$, then (\ref{flip2}), starting from $M_1$, satisfies conditions $ii)$ and $iii)$.  If $g\equiv 3\mod 4$, then only a weak Fano $M_j$ is reached, but condition $iii)$ still holds in this case. Since  $ P,P'$  may be  twisted, a twisted version of Bondal-Orlov's result is proven in Section  \ref{twisted}, from which Theorem \ref{A} follows.   Moreover,  we can use a twisted version of a result  due to Belmans, Fu and Raedschelders (see \cite[Theorem~A]{belmans}) to conclude that  some Hilbert powers of $X$ are also Fano or weak Fano visitors.
\begin{thmx}\label{B}
    Let $X$ be a K3 surface of Picard number 1 and let $d$ be an integer satisfying  $ 0\leq d\leq \frac{g+1}{4}$.  If  $g\not\equiv 3 \mod 4$ then $\Hilb^{d}X$ is a Fano visitor. If $g\equiv 3\mod 4$, then $\Hilb^{d}X$ is a weak Fano visitor.
\end{thmx}
We will prove Theorem \ref{A} and a more general version of Theorem \ref{B} in Section \ref{proof}. Note that Theorem \ref{B}  provides new families  of Fano visitors,  which are hyperkh\"aler varieties  of arbitrarily high dimension.

We believe that such sequence of standard flips should exist   for any smooth projective variety $X$ and  sufficiently positive  ample line bundle $ \Lambda$ such that $\omega_X\otimes \Lambda$ is very ample. In this context, a natural conjecture is the following: 
\begin{conjecture}(Tevelev \cite{Tevelev_2023})\label{conjecture}
    $M_1=\Bl_XM_0$ is log Fano.
\end{conjecture}
Since every weak Fano variety is log Fano, then Theorem \ref{A} implies the following.
\begin{corollary}
    Conjecture \ref{conjecture} holds when $X$ is a K3 surface of Picard number 1. 
\end{corollary} 
After proving Theorem \ref{A} and \ref{B}, we study the continuation of the sequence of flips (\ref{flip2}) for the cases  $g\equiv 0,3 \mod 4$.

When $g\equiv 3 \mod 4$, then (\ref{flip2}) continues  with a divisorial contraction $M_\nu\to M_{\nu+1}$, where $M_{\nu+1}$ is isomorphic to the projective space $\mathbb{P}^g$. Moreover, $M_\nu$ can be identified with the blow-up of $M_{\nu+1}$ along another K3 surface. In section \ref{section9} we prove that this K3 surface is dual to $X$ in the sense of Mukai (see \cite{Mukai_1999}). More precisely, it is isomorphic to the moduli space 
$$ \widehat{X}\coloneqq \mathcal{M}(2,h,(g-1)/2).$$
Using a normalized universal family, we construct an equivalence $\Theta:D^b(X)\to D^b(\hat{X})$ that sends the sequence of flips (\ref{flip2})—including the extra divisorial contraction $M_{\nu}\to M_{\nu+1}$— to the corresponding sequence of flips for $\hat{X}$, but in opposite order (see Proposition \ref{involutioncor} for the precise statement). The same holds when we apply $\Theta$  to the sequence of flops in (\ref{flops}).  This latter result explains the symmetry observed for $g=3$, and suggested for higher genus $g\equiv 3\mod 4 $ in \cite[Page~26]{arcara}. When $g=7$, the authors in \cite{cremona} proved that $\hat{X}$ is not isomorphic $X$. Following their ideas, we use the cohomological action of $\Theta$ to extends its result for any genus $g\equiv 3\mod 4$. 

Section \ref{sectiondiv}  focuses on the case when $g\equiv 0\mod 4$. In this setting, there is a divisorial contraction into a singular Fano manifold $\varphi:M_\nu\to \overline{M}$ which is induced by a stability condition $\overline{\sigma}$ on a divisorial wall $W$ with respect to $v$ \cite[Section~3.4]{Flapan_2021}. We prove the following Proposition, which  describes the fibers of this contraction and answers a question of Macr\`i.  This result also includes the description of the  fibers of the contraction $\varphi':\Omega_\nu\to \overline{\Omega}$, where $\Omega_\nu$ denotes the other component of $\Fix(\tau,\mathcal{M}_\nu)$ (see \cite[Section~5]{Flapan_2021}).
\begin{proposition}
    If  $F\in  \overline{M}$, then $\varphi^{-1}(F)$ is Lagrangian Grassmannian $\LGr(k,2k)$ for some $k$.  If $F\in  \overline{\Omega}$, then $\varphi'^{-1}$ is a  orthogonal Grassmannian isomorphic $\OGr(k,2k)$ for some $k$. 
\end{proposition}

\subsection{Relation to previous work} An analogy between the moduli space of Thaddeus stable pairs of a curve and the Moduli space of Bridgeland stable complexes with Mukai vector $v=(0,h,g-1)$ on a K3 surface was first observed by Arcara and Bertram in \cite[Section~8]{arcara}.  In this work, the authors described only the birational transformations $M_{i}\dasharrow M_{i+1}$ induced by rank 1 walls. Using the work of \cite{Flapan_2021} and \cite{mmp}, we are able to fully describe  the geometry of these birational maps  including all the higher-rank walls.  Another related work is by Martinez in \cite{martinez}, where the author constructed a sequence of flips as in (\ref{flips}) when $X=\mathbb{P}^2$ and $\Lambda=\mathcal{O}_X(d)$ for $d\geq 2$ (see \cite[Theorem~6.16]{martinez}). As in our case, he identified the blow-up of the Veronese embedding $X\hookrightarrow \mathbb{P}(\mathcal{O}(d)^*)$ with a subvariety of a Bridgeland moduli space of rank 0 complexes in $X$. The sequence (\ref{flips})  is then obtained via variation of the Bridgeland stability condition.  In both cases, the first flips that appear coincide with the flips described by Vermeire in \cite{Vermeire_2001} using secant varieties. We would also like to mention the work of Thaddeus on variation of the moduli of  parabolic Higgs bundles on a curve \cite{parabolics}. There, he showed that variation of the parabolic weights induces Mukai flops between different moduli spaces of  parabolic Higgs bundles. If we identified an ordinary parabolic bundle as a parabolic Higgs bundle with the zero Higgs field, then each moduli space of parabolic bundles is identified   as a subvariety of a moduli space of Higgs bundle. This sub-moduli space  is fixed by the antisymplectic involution given by  multiplying the  Higgs field by $-1$, and the restriction of each  Mukai flop  induces standard flips \cite[Section~7]{parabolicsflips} in the same way that the Bridgeland moduli spaces induce standard flips on the connected component of fixed locus by the antisymplectic involution $\tau$.

\subsection{Organization} The paper is organized as follows. In Section \ref{mukaiiso}, we recall the basics of the Mukai isomorphism. In Section \ref{sectio3}, we state a result proved in \cite[Section~5]{Flapan_2021} that describes the restriction of a Mukai flop to a connected component of the fixed locus of an antisymplectic involution. In Section \ref{section4}, we focus on studying eigenspaces corresponding to the action of (derived) dual functors on homomorphisms and their higher Ext groups. Section \ref{twisted} is dedicated to proving a twisted version of \cite[Theorem~A]{belmans}. We will use this result primarily to study how the derived category changes under the standard flips appearing  in (\ref{flip2}). In Section \ref{mukaiflop}, we examine the birational geometry of the moduli space $\mathcal{M}(v)$ and construct  the sequence of Mukai flops, closely following  the ideas  in \cite[Section~3]{Flapan_2021}. In Section \ref{constructionflips}, we study the geometry of the restriction of the Mukai flops from Section \ref{mukaiflop} to a connected component of the fixed locus of the antisymplectic involution. Here, we construct the sequence of birational maps (\ref{flip2}) and show that $M_1=\Bl_X\mathbb{P}^g$ and that each map $M_i\dasharrow M_{i+1}$ is small standard flip   for $i\geq 1$. This corresponds to condition $i)$. In Section \ref{section8}, we compute the ample cones  of each $M_i$. In Section \ref{proof}, we provide the proofs of Theorem \ref{A} and Theorem \ref{B}, along with some low genus examples.  In Section \ref{section9}, we restrict to the case $g\equiv 3\mod 4$. Here, we construct  an equivalence $\Phi:D^b(X)\to D^b(\hat{X})$ which interchanges the sequence of flips (\ref{flip2}) for $X$ and $\hat{X}$, and provide a proof that $\hat{X}$ is not isomorphic to $X$. Finally, in Section \ref{sectiondiv} we study the fibers of the contraction map $M_\nu\to \overline{M}$ in the case $g$ is divisible by 4, and show that these are symplectic Grassmannians.  

\subsection{Acknowledgements}
I would like to express my gratitude to my PhD advisor, Jenia Tevelev, for suggesting this problem to me and for his guidance and insightful comments. I would also like to thank  Eyal Markman,  Misha Verbitsky and Sokratis Zikas for useful  discussions, and especially Emanuele Macr\`i for  his insightful correspondence and discussions. This project has been partially supported by the NSF grants DMS-2401387 and DMS-2101726  (PI Jenia Tevelev). This paper was written during my visit to Instituto de Matem\'atica Pura e Aplicada in Rio de Janeiro, Brazil,  and Pontificia Universidad Cat\'olica de Chile (PUC). I am grateful to IMPA  and PUC for their  hospitality and to IMPA for its financial support. 

\section{Mukai isomorphism}\label{mukaiiso}
In this section, we recall the basic definitions and results regarding the Mukai isomorphism for moduli spaces of complexes on a K3  surface. We  will use these results to determine the restriction of line bundles on the ambient HK variety $\mathcal{M}_i$  to our component $M_i$ of the fixed locus of the antisymplectic involution $\tau$. We follow \cite{mmp} closely regarding Bridgeland stability conditions on K3 surfaces and birational geometry of their moduli spaces.

Let $X$ be a K3 surface and $v\in H_{\Alg}^*(X,\mathbb{Z})$ a primitive Mukai vector with $v^2\geq -2$. Let  $\sigma\in \Stab(X)$ be a generic stability condition with respect to $v$, meaning there are no strictly $\sigma$-semistable objects. Define 
 $$\mathcal{M}\coloneqq \mathcal{M}_\sigma(v)$$ to be the moduli space of $\sigma$-stable objects in $D^b(X)$ with Mukai vector $v$  (and some choice of a phase).

Let $T$ be a scheme. A complex $\mathcal{E}\in D^b(X\times T)$ is called a quasi-family for $\mathcal{M}$ (or simply a quasi-family), (see \cite[Appendix~2]{mukai} or \cite[Definition~4.5]{MR3194493}) if, for every $t\in T$, there is a point $E\in \mathcal{M}$ and some positive integer $\rho$ such that 
$$ \mathcal{E}|_t\simeq E^{\oplus \rho}.$$
If $T$ is connected, then $\rho$ does not depend on $T$ and is called the similitude of $\mathcal{E}$.  A quasi-family $\mathcal{E}\in D^b(X\times\mathcal{M})$ is called a quasi-universal family if, for every scheme $T$  and a quasi-family  $\mathcal{F}\in D^b(X\times T)$, there exists  a morphism $f:T\to \mathcal{M}$ such that $f^*\mathcal{E}$ and $ \mathcal{F}$ are equivalent, i.e., there are
vector bundles $\mathcal{V},\mathcal{W}$ on $T$ and an isomorphism
$$ f^*\mathcal{E}\otimes \pi_T^*\mathcal{V} \simeq \mathcal{F}\otimes \pi_T^*\mathcal{W}, $$
where $\pi_T$ is the projection $\pi_T:X\times T\to T$. A result of Mukai \cite[Theorem~A.5]{mukai} implies that  $\mathcal{M}$  always has  a quasi-universal family $\mathcal{E}\in D^b(\mathcal{M}\times X)$. The Mukai morphism induced by $\mathcal{E}$ is the linear map 
$$  \theta_{\mathcal{E}}:v^\perp_{\Alg} \to \NS(\mathcal{M})=N_1(\mathcal{M})^*,$$
such that for all smooth projective curves $C\subset \mathcal{M}$ and all $w\in v^\perp_{alg}$, we have
$$ \theta_{\mathcal{E}}(w)\cdot C= \frac{1}{\rho}\langle w,v(\Phi_\mathcal{E}\mathcal{O}_C)\rangle, $$
where $\Phi_\mathcal{E}:D^b(\mathcal{M})\to D^b(X)$ denotes the Fourier-Mukai functor induced by $\mathcal{E}$. Note that the condition $w\in v^\perp$ implies that $\theta_{\mathcal{E}}(w)$ is independent of the class of the family $\mathcal{E}$. Specifically, for any vector  bundle $\mathcal{V}$ on $\mathcal{M}$ and $\mathcal{E}'=\mathcal{E}\otimes\pi_\mathcal{M}^*\mathcal{V}$, then 
$$ \theta_{\mathcal{E}}(w) = \theta_{\mathcal{E}'}(w).$$
Indeed, we have $\Phi_\mathcal{E'}\mathcal{O}_C\simeq \Phi_\mathcal{E}\mathcal{V}|_C$. If $r$ denotes the rank of $\mathcal{V}$, then we can find a filtration  
$$ 0\subset \mathcal{V}_1 \subset \mathcal{V}_2 \subset \ldots\subset  \mathcal{V}_r=\mathcal{V}|_E, $$
with $\mathcal{L}_i=\mathcal{V}_{i}/\mathcal{V}_{i-1} $ line bundles. Thus we have 
$$ v(\Phi_{\mathcal{E}'}\mathcal{O}_C)=\sum_{i=1}^rv(\Phi_\mathcal{E}\mathcal{L}_i)=\sum_{i=1}^r  v(\Phi_\mathcal{E}\mathcal{O}_C)+\deg(\mathcal{L}_i)v=rv(\Phi_\mathcal{E}\mathcal{O}_C)+kv,$$
for some integer $k$.  Then the claim follows since $\mathcal{E}'$ has similitude $r\rho$.

We have the following result.
\begin{proposition}(\cite[Theorem~6.10]{MR3194493})
    $\theta_{\mathcal{E}}$ is an isomorphism of lattices.
\end{proposition}
\begin{remark}\label{mukaiflopdesc}
     Let  $\mathcal{M}\dasharrow\mathcal{M}'$ be a small birational map between two moduli spaces  obtained by crossing a wall on the wall-chamber structure of the space of stability conditions induced by $v$.  Since $\mathcal{M}$ and $\mathcal{M}'$ are isomorphic in codimension 1, we  obtain a canonical isomorphism 
    \begin{equation}\label{canonicaliso}
    \NS(\mathcal{M})\simeq \NS(\mathcal{M}').
    \end{equation}    
     Moreover, we can choose (quasi)-universal families $\mathcal{E}\in D^b(X\times \mathcal{M})$ and $\mathcal{E}'\in D^b(X\times \mathcal{M}')$ with the same similitude, such that  $\mathcal{E}\simeq \mathcal{E}'$ under the  birational map $\mathcal{M}\dasharrow \mathcal{M}'$ (see \cite[Page~733]{MR3194493} for details). Therefore  we have compatibility of the Mukai isomorphisms $\theta_{\mathcal{E}},\theta_{\mathcal{E}'}$, and the isomorphism (\ref{canonicaliso}).
\end{remark}

\section{Mukai flops and fixed locus of involutions}\label{sectio3}
In this section we state a result due to Flapan, Macr\`i, O'Grady and Sacc\`a in \cite[Section~5]{Flapan_2021}, which describes the restriction of a Mukai flop $\mathcal{M}\dasharrow\mathcal{M}'$ to a connected component of the fixed locus of an antisymplectic involution $\tau$. This result provides the local picture of the geometry of the birational maps $M_i\dasharrow M_{i+1}$ appearing in (\ref{flip2}). 

Let $\mathcal{M}$ be  an open subset (in the analytic or Zariski topology) of a HK manifold, and let $\mathcal{P}$ be a coisotropic subvariety of $\mathcal{M}$, which is isomorphic to a projective bundle (in the analytic or Zariski topology) $\mathbb{P}(\mathcal{V})\to \mathcal{Z}$ for some vector bundle $\mathcal{V}$ over a smooth manifold $\mathcal{Z}$. Assume that there is an antisymplectic involution 
$\tau$ acting on $\mathcal{M}$ which preserves the projective bundle $\pi:\mathcal{P}\to \mathcal{Z}$ and induces an involution  $\tau_\mathcal{Z}$  on $\mathcal{Z}$. We can linearize this action in order to obtain an action of $\tau$ on $\mathcal{V}$. Let $\varphi:\mathcal{M}\dasharrow \mathcal{M}'$ be the Mukai flop along $\mathcal{P}$, and let $\mathcal{P}'=\mathbb{P}(\mathcal{V}^\vee)\to \mathcal{Z}$  be  the exceptional locus of $\varphi^{-1}$. We also assume that there is an antisymplectic involution  $\tau'$ in $\mathcal{M}'$ that agrees with $\tau$ under the isomorphism $\mathcal{M}\setminus \mathcal{P}\simeq \mathcal{M}'\setminus \mathcal{P}'$, and restricts to the dual action on $\mathcal{P}'=\mathbb{P}(\mathcal{V}^\vee)$. 

\begin{proposition}(\cite[Proposition~5.2]{Flapan_2021})\label{mukaifloprest}
    Under the conditions of the previous paragraph, consider a connected component $\mathcal{F}\subset \Fix (\tau, \mathcal{M})$ which is not contained in $\mathcal{P}$ and let $\mathcal{F}'$ be its proper transform in $\mathcal{M}'$. Then the Mukai flop $\mathcal{M}\dasharrow\mathcal{M}'$ induces a birational map $\mathcal{F}\dasharrow\mathcal{F}'$ which factors as a series of disjoint standard flips (not necessarily small, i.e., possibly divisoral extractions or contractions), one for each connected component of $\mathcal{P}\cap \mathcal{F}$. More precisely, each  connected component $\Gamma\subset \mathcal{P}\cap \mathcal{F}$  can be identified with a projective bundle 
    $$  p:\Gamma\simeq\mathbb{P}(\mathcal{V}_\Gamma)\to \mathcal{W}, $$
    where $\mathcal{W}$ is a connected component of $\Fix(\tau_\mathcal{Z},\mathcal{Z})$, and $\mathcal{V}_\Gamma $ is one of the two eigenbundles $\mathcal{V}^+$ or $\mathcal{V}^-$  of $\mathcal{V}|_\mathcal{W}$, whose fibers over a closed point $w\in W$ are the  eigenspaces 
    $$ V^+=\{v\in \mathcal{V}|_{w}: \tau(v)=v\},\qquad V^-=\{v\in \mathcal{V}|_w: \tau(v)=-v\},$$
    respectively. Moreover if $\mathcal{V}_\Gamma'=\Ann(\mathcal{V}_\Gamma)\subset\mathcal{V}|_\mathcal{W}^\vee $, then $\mathcal{N}_{\Gamma/\mathcal{F}}\simeq \mathcal{O}_p(-1)\otimes p^*\mathcal{V}_\Gamma' $ and the flip corresponding to $\Gamma$ is the standard flip of $\mathcal{F}$  along  $\Gamma$ replacing $\Gamma\simeq\mathbb{P}(\mathcal{V}_\Gamma)$ by $\Gamma'\simeq\mathbb{P}(\mathcal{V}_\Gamma')$.
\end{proposition}
\section{Double dual functor and higher Ext groups}\label{section4}
Now we prove some preliminary results describing  the action of (derived) dual functors on homomorphisms and higher Ext's. The main result of this section is Corollary \ref{eigenspace} and Lemma \ref{lemmas4}, which we use  to describe the geometry of the standard flips appearing in Proposition \ref{mukaifloprest}. Here we will work in full generality, so $X$ will denote an arbitrary smooth projective variety. We will follow \cite{MR1804902} for our signs convention.  

Let $X$ be a smooth projective variety of dimension $n$. Fix a complex $R\in D^b(X)$ of vector bundles on $X$. We define the contravariant functor
$$ \Psi_{R}: D^b(X)^{op}\to D^b(X),\qquad \Psi_R(E)=\SheafHom(E,R).$$
 Let $E\in D^b(X)$ be another complex of vector bundles. For every integer $i$, there is   a canonical isomorphism 
\begin{equation}\label{morphim1}
    \Psi_R(E)[i]\overset{\sim}{\to} \Psi_{R[i]}(E)
\end{equation}
without the intervention of signs. More precisely, (\ref{morphim1}) is the morphism of complexes given in degree $p$ by the morphism of sheaves
$$ \prod_{r+s=p+i}\SheafHome(E^{-r},R^{s})\to  \prod_{r'+s'=p}\SheafHome(E^{-r'},R^{s'+i}) $$
obtained by using the identification $\SheafHome(E^{-r},R^{s})= \SheafHome(E^{-r'},R^{s'+i})$ if $(r,s)=(r',s'+i)$ and the zero map otherwise. We  also have a  canonical morphism
\begin{equation}\label{morphism3}
    \eta^R_E: E\to \Psi^2_{R}(E),
\end{equation}
which corresponds to the morphism of  complexes such that its degree $p$ part 
$$ E^p\to  \prod_{r+s+t=p}\SheafHome(\SheafHome(E^{r},R^{-s}),R^{t})$$ 
is obtained by  using the canonical morphism  
$$E^p\to \SheafHome(\SheafHome(E^p,R^{p+q}),R^{p+q}) $$
multiplied by $(-1)^{pq}$ if $r=p$ and $-s=t=p+q$, and zero otherwise. Analogously, we have the morphism 
$$\nu_{[i]}:\Psi^2_{R}(E)\to\Psi^2_{R[i]}(E),$$
which involves  the sign  $(-1)^{ip}$ in degree $p$. With these sign conventions, we have a commutative diagram
\begin{equation}\label{dobledual}
    \xymatrixcolsep{3pc}\xymatrix{
E\ar[r]^-{\eta^R_E} \ar[dr]_{\eta^{R[i]}_E} & \Psi^2_{R}(E) \ar[d]^{\nu_{[i]}}\\
 &\Psi^2_{R[i]}(E) 
 }
\end{equation}
 (see  \cite[Page~15]{MR1804902} for details). For each integer $i$ and $A,B\in D^b(X)$, we set 
 $$ \Ext^i(A,B)\coloneqq \Hom_{D^b(X)}(A,B[i]).$$
When $A,B$ are complexes of vector bundles, we obtain a natural linear map   $$\Psi^i_R:\Ext^i(A,\Psi_R(B))\to \Ext^i(B,\Psi_R(A)),$$ by sending a morphism $A\overset{\alpha}{\to} \Psi_R(B)[i]\simeq \Psi_{R[i]}(B)$ (using \ref{morphim1}), to the composition 
\begin{equation}\label{psiaction}
  \xymatrixcolsep{3pc}\xymatrix{
   B \ar[r]^{\eta^R_B} & \Psi_R^2(B) 
   \ar[r]^{\nu_{[i]}} & \Psi_{R[i]}^2(B) \ar[r]^-{\Psi_{R[i]}(\alpha)} &  \Psi_{R[i]}(A) =\Psi_R(A)[i].
  }  
\end{equation}  
Using the isomorphism (\ref{morphim1})  and the commutativity of (\ref{dobledual}), we obtain the following commutative diagram
\begin{equation}\label{square1}
    \xymatrix{
\Ext^i(A,\Psi_{R}(B))\ar[r]^{\Psi^i_R} \ar[d]^{\simeq} & \Ext^i(B,\Psi_{R}(A)) \ar[d]^\simeq\\
\Hom(A,\Psi_{R[i]}(B))\ar[r]^{\Psi^0_{R[i]}} & \Hom(B,\Psi_{R[i]}(A)).
 }
\end{equation}
 We will also use the following isomorphism coming from adjunction
$$ \Ext^i(B,\Psi_{R}(A))= \Hom(B,\Psi_{R[i]}(A)) \simeq \Hom(B\otimes A,R[i]).$$
Note that here we use the isomorphism  (\ref{morphim1}), which doesn't involve any sign. 
\begin{lemma}\label{lemma4.1}
    Let $A,B\in D^b(X)$ be two complexes of vector bundles. There is a commutative diagram
    \begin{equation}\label{square7}
        \xymatrix{
\Ext^i(A,\Psi_{R}(B))\ar[r]^{\Psi^i_R} \ar[d]^{\simeq} & \Ext^i(B,\Psi_{R}(A)) \ar[d]^{\simeq} \\
\Hom(A\otimes B,R[i] )\ar[r]^{\circ \mu} & \Hom(B\otimes A,R[i]),
 }
    \end{equation}
 where $\mu$  is the morphism of complexes $$\mu: B\otimes A\to A\otimes B,$$
 which sends   $a\otimes b\mapsto (-1)^{pq} b\otimes a$ for $a\in A^p$ and $b\in B^q$.
\end{lemma}
\begin{proof}
    By replacing $R$ by $R[i]$ and using the commutativity of  (\ref{square1}), it suffices to consider the
    case when $i=0$. Let $A,B\in D^b(X)$ be  two complexes of vector bundles, and let $\alpha:A\to \Psi_{R}(B) $ be a morphism. We may assume that $\alpha$ is given by an honest map of complexes. Let  $p,q$ be integers and denote by $\alpha^{pq}$ the map  of sheaves
    $$ \alpha^{pq}:A^p\to \SheafHome(B^q,R^{p+q})$$ 
    given by the composition
    $$ A^p\to \prod_{r+s=p}\SheafHome(B^{-r},R^{s})\to \SheafHome(B^q,R^{p+q}).$$
    By equation  (\ref{psiaction}), the map $\Psi_R^0(\alpha):B\to \Psi_{R}(A) $ is the composition 
    $$\xymatrixcolsep{4pc}\xymatrix{
    B \ar[r]^-{\eta^R_B} & \Psi_R^2(B) \ar[r]^-{\Psi_R(\alpha)} & \Psi_R(A), 
    }    $$
      and so there is a commutative diagram (now in  degree $q$)
      $$  \xymatrixcolsep{4pc}\xymatrix{
B^q \ar[d] \ar[rr] & & \prod_{r+s=q}\SheafHome(A^{-r},R^{s}) \ar[d]  \\
\SheafHome(\SheafHome(B^q,R^{p+q}),R^{p+q}) \ar[rr]^-{\SheafHome(\alpha^{pq}(-),R^{p+q})} & & \SheafHome(A^p,R^{p+q}),
 }$$
    which includes the sign $(-1)^{pq}$ on the left vertical map (see the definition of $\eta^R$). Since the vertical maps in (\ref{square7}), do not require the change of sign,  the lemma follows. 
\end{proof}
Since $X$ is smooth, c$\Psi_R^i:\Ext^i(A,\Psi(B))\to \Ext^i(B,\Psi(A))$ for any pair of elements $A,B\in D^b(X)$ by choosing resolution $\tilde{A}\to A,\tilde{B}\to B$ by complexes of vector bundles and applying (\ref{psiaction}) to the pair $\tilde{A},\tilde{B}$.  

When $A=B$,  we obtain an action of $\Psi_R$ on $\Ext^i(A,\Psi_R(A))$ by an involution. We write 
$$ \Ext^i(A,\Psi_{R}(A))^{\pm }=\{\alpha \in \Ext^i(A,\Psi_R(A)): \alpha^{\Psi_R}=\pm \alpha\}. $$
The main result of this Section is the following. 
\begin{corollary}\label{eigenspace}
   Let $A\in D^b(X)$ (not necessarily a complex of vector bundles). Then  there is an isomorphism 
    $$ \Ext^i(A,\Psi_{R}(A))^{+}\simeq \Ext^i(S^2A,R),  $$
    where $S^2A$ is  a derived symmetric square.\footnote{For a definition of $S^2A$ see, e.g.,  \cite[Definition~3.1]{derivedsym}.} 
\end{corollary}
\begin{proof}
    The statement follows from Lemma \ref{lemma4.1} when $A$ is a complex of vector bundles.  The general case follows since the  derived symmetric square is well defined i.e., doesn't depend (up to isomorphism) on the chosen  resolution $\tilde{A}\to A$.
\end{proof}
We also define the action of $\Psi_R$ on $\Ext^i(\Psi_R(A),A)$ via the Serre duality
$$ \Ext^i(\Psi_R(A),A)\simeq \Ext^{n-i}(A,\Psi_{R\otimes \omega_X}(A))^*,$$
obtained by using  the dual action of $\Psi_{R\otimes \omega_X}$  on the right hand side. Then the  eigenbundles satisfy the relations
\begin{equation}\label{eigenrelations}
    \Ext^i(\Psi_R(A),A)^{\pm}=\Ann(\Ext^{n-i}(A,\Psi_{R\otimes \omega_X}(A))^{\mp}). 
\end{equation}
Indeed, this is just a consequence of the general fact that  if $\tau $ is an involution on a vector space $V$, and $V^*$ is endowed with the dual action, then 
$$ V^\pm =\Ann ((V^*)^\mp).$$

We will also use the following useful fact. Assume that $\eta^R_E:E\to \Psi_R^2(E)$ is an isomorphism  for every $E$ (e.g, if $R$ is a shifted  line bundle in $X$). Let $\alpha\in \Ext^1(T,\Psi_R(T))$  such that $\alpha^{\Psi_R}\in \{\pm \alpha\}.$ Then $E$ is $\Psi$-invariant and so there is an isomorphism $t:E\to \Psi_R(E)$. We relate the action of $\Psi_R$ on $t$ and $\alpha$ as follows. Consider a morphism of exact triangles 
\begin{equation}\label{tsaction}
\resizebox{0.91\hsize}{!}{%
    $\xymatrixcolsep{5.5pc}\xymatrix{ \Psi_R(A)\ar[r]^{u} \ar[d]^r & E  
\ar[d]^t \ar[r]^v & A \ar[r]^{\alpha} \ar[d]^{\eta_A^R} & \Psi_R(A)[1]  \ar[d]^{r[1]} \\ 
 \Psi_R(A)\ar[r]^{\Psi_R(v)} & \Psi_R(E)  \ar[r]^{\Psi_R(u)} & \Psi^2(A) \ar[r]^{\Psi_{R[1]}(\alpha)\circ\nu_{[1]}} & \Psi_R(A)[1]. }$%
 }
\end{equation}
We have the following result.
\begin{lemma}\label{lemmas4}
    $t^{\Psi_R}=\pm t\iff \alpha^{\Psi_R}=\pm \alpha$.
    \begin{proof}
    By comparing the middle square of (\ref{tsaction}) with the following commutative diagram
    \begin{equation*}
        \xymatrix{
        E\ar[r]^v \ar[d]^{\eta_E^R}& A \ar[d]^{\eta_A^R} \\
        \Psi_R^2(E) \ar[d]^{\Psi_R(t)} \ar[r]^{\Psi_R^2(v)} &\Psi_R^2(A) \ar[d]^{\Psi_R(r)} \\
        \Psi_R(E) \ar[r]^{\Psi_R(u)} & \Psi^2_R(A),
        }
    \end{equation*}
    we obtain that $t^{\Psi_R}=\pm t \iff \Psi_R(r)=\pm \Id_{\Psi^2_R(A)}\iff r=\pm \Id_{\Psi_R(A)}.$ Since $\alpha^{\Psi_R}=\Psi_{R[1]}(\alpha)\circ \nu_{[1]}\circ \eta_A^R$, from the right square of (\ref{tsaction}) we conclude that $r=\pm \Id_{\Psi_R(A)}\iff\alpha^{\Psi_R}=\pm \alpha$ and the Lemma follows.
    \end{proof}
\end{lemma}
\section{Twisted projective bundles and flips}\label{twisted}
The presence of non-fine moduli spaces in the sequence of Mukai flops in (\ref{flops}) implies that the sequence of standard flips  (\ref{flip2})  may involve twisted projective bundles.  In this section, we prove a twisted version of a result due to Belmans, Fu and Raedschelders (see  \cite[Theorem~A]{belmans}) which describes how the derived category changes under a standard flip. We refer to C\u ald\u araru's book \cite{caldararu} and Bernardara's work \cite{BERNARDARA2016181}  for the standard  material on twisted sheaves and twisted derived categories. Since this section is just an adaptation of the proof of \cite[Theorem~A]{belmans} for the twisted case, we will follow its notation closely.

We begin with the precise definition of a twisted standard flip. Our spaces are assumed to be smooth compact complex  manifolds but not necessarily projective varieties.

Let $X$ be a smooth compact  complex  manifold containing a twisted projective bundle $$\pi:Z=\mathbb{P}(\mathcal{V})\to F,$$  where $\mathcal{V}$ is a twisted vector bundle on a smooth compact complex manifold $F$. Let $\alpha\in \Br F$ be the Brauer class associated to $\mathcal{V}$. We denote by $\mathcal{O}_{\pi}(-1)$  the tautological   $\pi^*\alpha$-twisted line bundle on $Z$. Assume that the normal bundle of $Z$ in $X$ satisfies  
$$ \mathcal{N}_{Z/X}|_{\mathbb{P}^k}=\mathcal{O}_{\mathbb{P}^k}(-1)^{\ell+1} $$
for every fiber $\mathbb{P}^k=\pi^{-1}(x)$ of $\pi:Z\to F$, or equivalently, there is an isomorphism $$\mathcal{N}_{Z/X}\simeq \mathcal{O}_{\pi}(-1)\otimes \pi^*\mathcal{V}',$$
where $\mathcal{V}'$ is some $\alpha^{-1}$-twisted vector bundle of rank $\ell+1$ on $F$ (see \cite[Proposition~1.2.10]{caldararu} for the notion of a product of two twisted sheaves).

Let $\tau:\overset{\sim}{X}\to X$ be the blowup of $X$ along $Z$ with exceptional divisor $j:E\hookrightarrow \overset{\sim}{X}$. We have isomorphism an $E\simeq \mathbb{P}(\mathcal{V})\times_Z\mathbb{P}(\mathcal{V}')$.  We will assume that there is a contraction morphism  $\tau':\overset{\sim}{X}\to X'$ to a smooth compact complex manifold $X'$, which restriction to  $E$ is identified with the projection to the second factor $E\to \mathbb{P}(\mathcal{V}')$, and $\tau':\tilde{X}\to X'$ is identified with the blowup of $X'$ along $Z'\coloneqq\mathbb{P}(\mathcal{V}')$.  Then the normal bundle $\mathcal{N}_{Z'/X'}$ satisfies 
$$\mathcal{N}_{Z'/X'}|_{\mathbb{P}^\ell}=\mathcal{O}_{\mathbb{P}^\ell}(-1)^{k+1}$$
for every fiber $\mathbb{P}^\ell$ of $\pi':Z'\to F$. In summary, we have the following diagram
\begin{equation}\label{twistedflip}
    \xymatrix{
  & &  E\ar[ddrr]^{p'} \ar[ddll]_{p} \ar@{^{(}->}[d]^j  &  &\\
    &   &  \overset{\sim}{X} \ar[dl]_{\tau} \ar[dr]^{\tau'}  &  & \\
Z\ar[drr]_{\pi} \ar@{^{(}->}[r]^{i}  &  X \ar@{-->}[rr]^{\phi} & &  X' & Z' \ar@{_{(}->}[l]_{i'}  \ar[dll]^{\pi'}\\
    &  &  F & &  }.
\end{equation}
We refer to (\ref{twistedflip}) as the twisted standard flip diagram of $X$ along $Z$. For each $m\in \mathbb{Z}$, the tensor product $\mathcal{O}_{\pi}(m)=\mathcal{O}_{\pi}(1)^{\otimes m}$ defines a $\pi^*(\alpha^{-m})$-twisted line bundle on $Z$ and so by C\u ald\u araru's work \cite[Definition~3.1.1]{caldararu}, we   have a well defined functor
$$ \Phi_m:D^b(F,\alpha^{m})\to D^b(X),\qquad \mathcal{E}\mapsto i_*(\pi^*(\mathcal{E})\otimes \mathcal{O}_\pi(m)).$$
The next result is a twisted version of \cite[Theorem~A]{belmans}.
\begin{proposition}\label{twistedemd}
\begin{enumerate}[i)] Assume that $k\geq\ell$, 
    \item For every integer $m\in \mathbb{Z}$, the functor $\Phi_m$
    is fully faithful. 
    \item We have the following semiorthogonal decomposition of $D^b(X):$
    \begin{equation}\label{iidec}
        D^b(X)=\langle   \Phi_{-k+\ell}(D^b(F,\alpha^{-k+\ell})),\ldots, \Phi_{-1}(D^b(F,\alpha^{-1})),\tau_*\tau'^*D^b(X')\rangle.
    \end{equation}
\end{enumerate}
\end{proposition}
Part $i)$ is a consequence of the following proposition, which is a twisted version of \cite[Proposition~3]{belmans}. Consider the following diagram 
$$\xymatrix{
 Y\ar[d]^\pi  \ar@{^{(}->}[r]^{i}  &   M \\
  W & },   $$
where $W,Y,M$ are smooth compact complex manifolds, $\pi:Y\to W$ is a smooth proper morphism and $i:Y\hookrightarrow M$ is a closed immersion. For a closed point $x\in W$, we denote by $P=\pi^{-1}(x)$ the fiber of $\pi$ at $x$ and $j:P\hookrightarrow X$  its closed immersion obtained by composing with $i$.  
\begin{proposition}\label{pro3}
    Assume that for every fiber $P$ of $\pi$ and all integers $p,q$ with $p+q>0$ 
    $$ H^p(P, \wedge^q\mathcal{N}')=0, $$
where $\mathcal{N}'\coloneqq \mathcal{N}|_{Y/M}|_P $. Let $\beta\in \Br W$.  Then for every $\pi^*\beta$-twisted line bundle $\mathcal{L}$ on $Y$,  the functor 
$$\Phi: D^b(W,\beta^{-1})\to D^b(M),\qquad \mathcal{E}\mapsto i_*(\pi^*(\mathcal{E})\otimes \mathcal{L})  $$
is fully faithful.
\end{proposition}
We use a twisted  Bondal-Orlov criterion  for fully faithful functors due to C\u ald\u araru.  
\begin{theorem}(\cite[Theorem~3.2.1]{caldararu}). Let $X,Y$ be smooth compact complex manifolds and $\alpha\in \Br(X) $. Let $\mathcal{E}\in D^b(X\times Y,p^*(\alpha^{-1}))$ where $p:X\times Y\to X$ and let  $\Phi_{\mathcal{E}}:D^b(X,\alpha)\to D^b(Y)$ be its associated Fourier-Mukai transform. Then $\Phi_\mathcal{E}$ is fully faithful if and only if for all closed points $x,x'\in X$, we  have:
\begin{enumerate}[(i)]
    \item $\Hom_Y(\Phi_\mathcal{E}(\mathcal{O}_x),\Phi_\mathcal{E}(\mathcal{O}_x))\simeq \mathbb{C}$.
    \item $\Hom_Y(\Phi_\mathcal{E}(\mathcal{O}_x),\Phi_\mathcal{E}(\mathcal{O}_x)[n])=0 $ for all $m\not\in[0,\Dim X]$.
    \item $\Hom_Y(\Phi_\mathcal{E}(\mathcal{O}_x),\Phi_\mathcal{E}(\mathcal{O}_{x'})[m])=0$ for all $m\in\mathbb{Z} $ and $x\neq x'$. 

\end{enumerate}
    
\end{theorem}
\begin{proof}(Proposition \ref{pro3})
Let $x\in W$ be a closed point write $P=\pi^{-1}(x)$. The proof of the vanishing $(i),(ii)$ and $(iii)$ follows just as in the untwisted case since $\pi^*(\mathcal{O}_x)\otimes \mathcal{L}= \mathcal{L}|_P$ defines an untwisted line bundle on $P$ (see the proof of \cite[Proposition~3]{belmans}). 
\end{proof}
In order to prove  part $ii)$ of Proposition \ref{twistedemd}, we replace the categories $\mathcal{A}(a,b)$ of \cite[(30)]{belmans} by 
    \begin{equation}\label{Aab}
        \mathcal{A}(a,b)=j_*((\pi\circ p)^*(D^b(F,\alpha^{-a+b}))\otimes \mathcal{O}(a,b)),
    \end{equation}
    where $$\mathcal{O}(a,b)\coloneqq p^*\mathcal{O}_p(a)\otimes p'^*\mathcal{O}_{p'}(b).$$  Note that $\mathcal{O}(a,b)$ is a $(\pi\circ p)^*(\alpha^{-a+b})$-twisted line bundle on $E$ and so $\mathcal{A}(a,b)\subset D^b(\tilde{X})$. By Bernardara's  result \cite[Theorem~5.1]{BERNARDARA2016181}, we obtain for every $m\in \mathbb{Z}$, semiorthogonal decompositions 
    \begin{equation}\label{bernd1}
        \resizebox{0.91\hsize}{!}{%
         $\mathcal{A}(a,\star)\coloneqq j_*( p^*\mathcal{O}_p(a)\otimes p'^*(D^b(Z',\pi'^*(\alpha^{a})) ))=\langle \mathcal{A}(a,m-\ell),\mathcal{A}(a,m-\ell+1),\ldots ,\mathcal{A}(a,m)\rangle,$%
         }
    \end{equation}
    and 
    \begin{equation}\label{bernd2}
        \resizebox{0.91\hsize}{!}{%
        $\mathcal{A}(\star,b)\coloneqq j_*( p'^*\mathcal{O}_{p'}(b)\otimes p^*(D^b(Z,\pi^*(\alpha^{-b})) ))=\langle \mathcal{A}(m-k,b),\mathcal{A}(m-k+1,b),\ldots ,\mathcal{A}(m,b)\rangle.$%
        }
    \end{equation}
    By applying Orlov's blowup formula \cite[Theorem~4.3]{orlov}
    to the blowups $\tau:\tilde{X}\to X$ and $\tau':\tilde{X}\to X'$, we obtain the following semiorthogonal decompositions of $D^b(\tilde{X}$
    \begin{equation}\label{dbtilde1}
        \begin{split}
            D^b(\tilde{X})&=\langle \mathcal{A}(\star,-\ell),\mathcal{A}(\star,-\ell+1),\ldots , \mathcal{A}(\star,-1),\tau^*D^b(X)\rangle,  \\ D^b(\tilde{X})&=\langle \mathcal{A}(-k,\star),\mathcal{A}(-k+1,\star),\ldots , \mathcal{A}(-1,\star),\tau'^*D^b(X')\rangle.
        \end{split}
    \end{equation}
    In the untwisted case, the semi-orthogonality of the blocks appearing in (\ref{iidec}) follows from  \cite[Proposition~8]{belmans} and \cite[Proposition~9]{belmans}. The proof of these results for the twisted case is done in the same way by using the standard relations between derived  functors of twisted, i.e.,  base change, projection formula and adjoint functors (see \cite[Section~2.3]{caldararu}), and introducing Brauer classes whenever needed. Indeed, we start with the following vanishing result (see \cite[Lemma~7]{belmans})
    $$ \RHom_{\tilde{X}}(\mathcal{A}(a_1,b_1),\mathcal{A}(a_2,b_2))=0,$$ which holds  for any  of the following cases:
    \begin{enumerate}
        \item $1\leq a_1-a_2\leq k-1$.
        \item $1\leq b_1-b_2\leq \ell -1$.
        \item $a_1-a_2=k$ and $0\leq b_1-b_2\leq \ell-1$.
        \item $b_1-b_2=\ell$ and $0\leq a_1-a_2\leq k-1$.
    \end{enumerate}
    The proof of this result reduces to show the vanishing  of \cite[(44)]{belmans}
    $$ \RHom_F(\mathcal{F}_1,\mathcal{F}_2\otimes\pi_*\mathcal{O}(a_2-a_1)\otimes \pi_*'\mathcal{O}_{\pi'}(b_2-b_1))=0,$$
    where $\mathcal{F}_i\in D^b(F,\alpha^{a_i-b_i})$,  which follows by working locally in $F$ (in the analytic topology). Similarly,  the vanishing of \cite[(45)]{belmans} 
    $$ \RHom_X(\Phi_m(\mathcal{E}),\Phi_{m'}(\mathcal{F}))=0,\qquad \mathcal{E}\in D^b(F,\alpha^{m}),\mathcal{F}\in D^b(F,\alpha^{m'}),$$
 where  $m,m'$ are integers satisfying $0<m-m'<k-\ell $, reduces to show the vanishing of \cite[(48)]{belmans} which is proving  by working locally in $Z$.  The vanishing 
$$ \RHom_X(\tau_*\tau'^*(\mathcal{F}),\Phi_m(\mathcal{G}))=0 $$  of \cite[Proposition~9]{belmans}, follows   similarly by using the twisted version of the categories $\mathcal{A}(a,b)$ in (\ref{Aab}) and the semiorthogonal    decompositions    (\ref{bernd1}), (\ref{bernd2})  (\ref{dbtilde1}).

The  generation of the blocks in (\ref{iidec}) also follows without change by working with categories (\ref{Aab}) and the semiorthogonal decomposition (\ref{bernd1}), (\ref{bernd2}) and (\ref{dbtilde1}).
\begin{remark}\label{uuntwisteembd}
    By mutating  some blocks  to the right in (\ref{iidec}), we obtain a semiorthogonal decomposition 
    \begin{multline}
    D^b(X)=\langle \Phi_{-m}(D^b(F,\alpha^{-m}),\ldots,\Phi_{-1}(D^b(F,\alpha^{-1})),\tau_*\tau'^*D^b(X),\Phi_0(D^b(F)), \\
    \ldots,\Phi_{k-\ell-m-1}(D^b(F,\alpha^{k-\ell-m-1}))\rangle
\end{multline}
    for any integer $m$ with $0\leq m\leq k-\ell$. Thus if $k-\ell>0$,  there is an embedding $D^b(F)\hookrightarrow D^b(X)$.    
\end{remark}
\section{Construction of the sequence of flops}\label{mukaiflop}
Let $X$ be a K3 surface of Picard number 1 and genus $g\geq 4$. Let $\Lambda$ be the ample generator of $\Pic X$ and define 
$$ M_0=\mathbb{P}(H^0(X,\Lambda)^*),\qquad M_1=\Bl_XM_0,$$
 where $X\hookrightarrow M_0$ is embedded by the linear system $|\Lambda|$. Note that $\Lambda$ is very ample since $g\geq 4$ (see, e.g., \cite[Pages~40-41]{k3surface}). In order to obtain the sequence of  flips (\ref{flips})  satisfying the three conditions, we start by recalling a construction of a sequence of Mukai flops involving moduli space of Bridgeland stable objects of  $D^b(X)$ due to Flapan, Macr\`i, O'Grady and Sacc\`a (see \cite[Section~3]{Flapan_2021}). Some of the  results  there are proven under the assumption  that $g$ is divisible by 4 but the same ideas can be applied in order to obtain  results valid for  any genus $g\geq 4$. Everything  said here is essentially  done in \cite[Section~3]{Flapan_2021}.

Consider the Mukai vector 
\begin{equation}\label{v}
     v=(0,h,1-g)\in H_{\Alg}^*(X,\mathbb{Z}),
\end{equation}
where $h=c_1(\Lambda)$ and let $\mathcal{M}=\mathcal{M}_\Lambda(v)$ be the moduli space of Gieseker stable sheaves with Mukai vector $v$. By Section \ref{mukaiiso}, we have the Mukai isomorphism of lattices $$\theta\coloneqq \theta_\mathcal{E}:v^\perp_{\Alg}\to \NS(\mathcal{M}),$$  which is induced by some (quasi)-universal family $\mathcal{E}\in D^b(X\times \mathcal{M})$. Generically, $\mathcal{M}$ parametrizes sheaves of the form 
$$ i_{C,*}(\zeta),$$
where $\zeta\in \Pic^0(C)$ is a line bundle of degree zero and $C$ is a smooth curve in the linear system $|\Lambda|$ with embedding $i_C:C\hookrightarrow X$. There is a Lagrangian fibration 
$$ \pi:\mathcal{M}\to  |\Lambda|= \mathbb{P}^g ,\qquad i_{C,*}(\zeta)\mapsto C,$$
which has the zero section $\mathbb{P}^g\hookrightarrow \mathcal{M} $ given by $C\mapsto i_{C,*}(\mathcal{O}_C)$. 

We define the divisors 
$$ f:=\pi^*\mathcal{O}_{\mathbb{P}^g(1)}=\theta(0,0,-1),\qquad \lambda\coloneqq   \begin{cases}
\theta(2,-h,g/2-1) & \text{ if  $g$ is even, }  \\
\theta(2,-h,(g-1)/2) & \text{ if  $g$ is odd. } 
\end{cases}$$
Now we describe the movable cone of $\mathcal{M}$ and its chamber decomposition. We start with some preliminary notations. Let $J$ be the set of pair of integers $(c,d)$ with $c\geq 0,d\geq -1$ satisfying 
\begin{enumerate}[a)]
    \item $\frac{(g-1)c^2-d}{2c+1}\in \mathbb{Z}$.
    \item $\mu(c,d)> 0$, where 
        \begin{equation}\label{mucd}
             \mu(c,d) =\begin{cases}
\frac{g-1-4d-(2c+1)^2}{2(2c+1)^2} & \text{ if  $g$ is even,}  \\
\frac{g-1-4d}{2(2c+1)^2} & \text{ if  $g$ is odd. } 
\end{cases}
        \end{equation}
\end{enumerate}
Let $I= \{\mu(c,d) \in (c,d)\in J\}$ be the set of possibles values of $\mu$. We call $I$ the set of slopes and  ordered it by decreasing value of $\mu$. Denote $\eta=\#I-2$. We  will identify $I$ with the  set $\{-1,0,\ldots,\eta\}$ and its usual order.  For each integer $-1\leq i\leq \eta$, $\mu_i\in I$ denotes its corresponding slope and $J^i=\{(c,d)\in J:\mu(c,d)=\mu_i\}$ is the set of pairs with slope $\mu_i$. For example,  $J^{-1}=\{(0,-1)\}$ and $J^0=\{(0,0)\}$.

For each $i\in I$, we define  the divisor $$\tilde{a}_i=\lambda+\mu_if\in \NS(\mathcal{M}).$$ For each pair $(c,d)\in J$, the Mukai vector  $v_{c,d}$ is given by the formula
\begin{equation}\label{vi}
    v_{c,d}=\left(2c+1,-ch, \frac{(g-1)c^2-d}{2c+1}\right)\in H_{\Alg}^*(X,\mathbb{Z}).
\end{equation}
\begin{lemma}(See \cite[Lemma~3.18]{Flapan_2021}])
 The movable  cone of $\mathcal{M}$ is 
 $$\Mov \mathcal{M}= \mathbb{R}_{\geq 0}f + \mathbb{R}_{\geq 0}\lambda. $$
Moreover, the chamber decomposition for $\Mov M$ is given by the rays generated by the divisors $\tilde{a}_i$ for $i\in I$.
\end{lemma}
\begin{proof}
    The proof is analogous to \cite[Lemma~3.18]{Flapan_2021} using \cite[Theorem~3.9]{Flapan_2021}. Note that when $g$ is odd, there is no ray coming from \cite[Theorem~3.9~(ii)]{Flapan_2021} . Thus the movable cone coincides with the positive cone.
\end{proof}
As in \cite[Section~3.2]{Flapan_2021}, there is a one parameter  family of Bridgeland stability conditions for $D^b(X)$: $$\sigma_{\alpha}\coloneqq\sigma_{\alpha,-1/2}=(Z_{\alpha,-1/2},\Coh^{-1/2}X),\qquad \alpha\in (\alpha_0,\infty), $$ where 
$$ \alpha_0= \begin{cases}
\frac{1}{2\sqrt{g-1}} & \text{ if  $g$ is even,}  \\
0 & \text{ if  $g$ is odd. } 
\end{cases} $$
Recall that for a real number $\beta\in \mathbb{R}$,  $\Coh^{\beta}X$ denotes the abelian sub-category of $D^b(X)$:
$$\Coh^{\beta}X\coloneqq\langle \mathcal{T}^{\beta}, \mathcal{F}^\beta[1] \rangle,$$
where 
\begin{equation}\label{definitioncohb}
    \begin{split}
     \mathcal{T}^{\beta}&=\{E\in\Coh X, \forall E \twoheadrightarrow Q:\text{ torsion free }:\mu(Q)>\beta \}, \\
      \mathcal{F}^{\beta}&=\{E\in\Coh X, E \text{ torsion free and } \forall 0\neq K\hookrightarrow E: \mu(K)\leq \beta \},
\end{split}
\end{equation}
and $\mu(E)\coloneqq(c_1(E),h)/h^2$. By the  work of Bayer and Macr\`i  in \cite{MR3194493}, there is piece-wise analytic map
$$\ell:\Stab^\dag(X)\to \NS(\mathcal{M}),\qquad \sigma\mapsto \ell_\sigma,$$
where $\Stab^\dag (X)$ is the connected component of the space of stability conditions $\Stab(X)$ containing $\sigma_{\alpha}$ (see \cite[Section~11]{stability} for a more precise description).
Using \cite[Lemma~9.2]{MR3194493}, we obtain  (up to multiplication by  a positive scalar) that
\begin{equation}\label{limits}
    \begin{split}
       &\ell_{\sigma_{\alpha}}\to  f  \quad \text{as}\quad  \alpha\to \infty, \\  &\ell_{\sigma_{\alpha}}\to  \lambda \quad \text{as}\quad  \alpha\to \alpha_0^+. 
    \end{split}
\end{equation}
 Thus, for each $i\in I$, there is an  $\alpha_{i}\in (\alpha_0,\infty)$ such that  $\ell_{\sigma_{\alpha_i}}= \tilde{a}_i$ (up to multiplication by a positive scalar). Let $1\gg \varepsilon >0$ be sufficiently small. We have the following result.  
\begin{lemma}(See \cite[Lemma~3.19]{Flapan_2021})
    For each $i\in I$, the ray generated by the divisor $\tilde{a}_i$ induces a birational map $\mathcal{M}_{\sigma_{\alpha_i+\varepsilon}}(v)  \dasharrow \mathcal{M}_{\sigma_{\alpha_i-\varepsilon}}(v)$, which factors as a series of disjoint Mukai flops, one for each pair $(c,d)\in I$ with slope $\mu(c,d)=\mu_i$. More precisely, each pair $(c,d)\in J^i$, induces a  Mukai flop, whose exceptional locus is identified with a (possibly twisted)  projective bundle $\mathcal{P}_{c,d}=\mathcal{P}(\mathcal{V}_{c,d})$, which parametrizes complexes $E\in \mathcal{M}_{\sigma_{\alpha_i+\varepsilon}}(v) $ such that their   $\sigma_{\alpha_i-\varepsilon}$-destabilizing sequence is given by a non-trivial  exact triangle 
    $$ K\to  E\to Q,$$
    for some elements $K\in \mathcal{M}(v_{c,d})$ and $Q\in \mathcal{M}(v-v_{c,d}).$
\end{lemma}
\begin{proof}
    The proof is the same as in \cite[Lemma~3.19]{Flapan_2021}  using \cite[Example~3.11]{Flapan_2021}. Note that for two distinct  pairs $(c,d),(c',d')\in I$ with slope $\mu_i$, the projective bundles $\mathcal{P}_{c,d}$ and $\mathcal{P}_{c',d'}$ are disjoint by the uniqueness of the Harder-Narasimham filtration (with respect $\sigma_{\alpha_i-\varepsilon}$). Thus, the Mukai flops are disjoint.  
\end{proof}
For each $i\in I$, we define $$\mathcal{M}_i\coloneqq\mathcal{M}_{\alpha_{i}+\varepsilon}(v),\qquad \mathcal{M}_{\eta+1}\coloneqq \mathcal{M}_{\sigma_{\eta}-\varepsilon}(v).$$
If $\alpha\gg 0$, by the large volume limit property \cite[Section~14]{stability}, we have that $\mathcal{M}=\mathcal{M}_{\sigma_\alpha}(v)=\mathcal{M}_{-1}$. Thus, by tanking  $\alpha\to \alpha_0^+$,  we obtain a sequence of flops between moduli spaces
\begin{equation}\label{flops2}
  \xymatrix{
 \mathcal{M}=\mathcal{M}_{-1}  \ar@{-->}[r] & 
\mathcal{M}_0  \ar@{-->}[r] &    \ldots   \ar@{-->}[r]  &  \mathcal{M}_\eta \ar@{-->}[r]  &  \mathcal{M}_{\eta+1},  }
\end{equation}
where each flop  $\mathcal{M}_i\dasharrow \mathcal{M}_{i+1}$ fits into a diagram 
 \begin{equation}\label{flop1}
    \resizebox{0.91\hsize}{!}{%
    $\xymatrixcolsep{0.6pc}\xymatrix{
   \underset{(c,d)\in J^i}{\coprod}\mathcal{P}_{c,d}\ar[ddrr] \ar@{^{(}->}[r]  &  \mathcal{M}_{i} \ar@{-->}[rr] \ar[dr]  & &  \mathcal{M}_{i} \ar[dl] &  \underset{(c,d)\in J^i}{\coprod}\mathcal{P}_{c,d}'\ar@{_{(}->}[l]  \ar[ddll]\\
   & &  \overline{\mathcal{M}}_{i}  & & 
   \\    &  &    \underset{(c,d)\in J^i}{\coprod}\mathcal{M}(v_{c,d}) \times \mathcal{M}(v-v_{c,d}) \ar@{_{(}->}[u]    & &  }$%
   }
\end{equation}
Here  $\mathcal{P}_{c,d}=\mathbb{P}(\mathcal{V}_{c,d}), \mathcal{P}_{c,d}'=\mathbb{P}(\mathcal{V}_{c,d}^*)$, where  $\mathcal{V}_{c,d},\mathcal{V}_{c,d}^*$  are (possibly twisted) vector bundles over the product   $\mathcal{M}(v_{c,d})\times \mathcal{M}(v-v_{c,d})$ of moduli spaces of $\sigma_\alpha$-stable objects\footnote{These moduli spaces doesn't depends on the value of $\alpha$ (see \cite[Example~3.11]{Flapan_2021}).}, whose  fibers over a point $(T,T')\in \mathcal{M}(v_{c,d})\times \mathcal{M}(v-v_{c,d}) $ are given by 
$$  \mathcal{V}_{c,d}|_{(T,T')}=\Ext^1(T',T)\qquad \mathcal{V}_{c,d}^*|_{(T,T')}=\Ext^1(T',T)^*=\Ext^1(T,T').$$ 
Locally around  $\mathcal{P}_{c,d}$, $\mathcal{M}_i\dasharrow \mathcal{M}_{i+1}$ is a Mukai flop replacing $\mathcal{P}_{c,d}$ by $\mathcal{P}_{c,d}'$. The contraction  morphisms
\begin{equation}\label{contmorphims}
    \xymatrix{
\mathcal{M}_{i} \ar[dr]   \ar@{-->}[rr] &  &\mathcal{M}_{i+1} \ar[dl] \\
& \overline{\mathcal{M}}_{i}  &  }
\end{equation}
are obtained by identifying $S$-equivalent classes of $\sigma_{\alpha_{i}}$-semistable.
\section{Construction of the sequence of flips}\label{constructionflips}
Consider the contravariant functor 
\begin{equation}\label{psi} \Psi:D^b(X)^{op}\overset{\sim}{\to} D^b(X),\qquad E\mapsto \SheafHom(E,\Lambda^*[1]).
\end{equation}
Then $\Psi$ induces an antisymplectic involution $\tau$ on each $\mathcal{M}_i$ (see the proof of \cite[Proposition~3.1]{Flapan_2021} and \cite[Lemma~3.24]{Flapan_2021}).
In $\mathcal{M}$, the fixed locus of $\tau$ has two connected components
$$\Fix(\tau,\mathcal{M})= \mathbb{P}^g\cup \Omega,$$
where $\mathbb{P}^g$ corresponds to the  image of the zero section of the Lagrangian fibration $\pi:\mathcal{M}\to \mathbb{P}^g$, and $\Omega$ is the closure of the locus parametrizing sheaves of the form $i_{C,*}(\zeta)$, where  $C\in |\Lambda|$ is a smooth curve and  $\zeta$ is a non-trivial square root of $\mathcal{O}_C$ (see  \cite[Proposition~4.1]{Flapan_2021}).
The exceptional locus of $\mathcal{M}\dasharrow \mathcal{M}_0$ coincides with $\mathbb{P}^g$ and so the fixed locus of $\tau$ in $\mathcal{M}_0$ is 
$$ \Fix(\tau, \mathcal{M}_0)=M_0\cup \Omega,$$
where $M_0\coloneqq (\mathbb{P}^g)^\vee=\mathbb{P}(H^0(X,\Lambda)^*)$ (see \cite[Example~3.21]{Flapan_2021}). When $g$ is divisible by 4, the authors showed in  \cite[Section~5]{Flapan_2021} that $\Fix(\tau,\mathcal{M}_i)$   has two connected components $M_i,\Omega_i$ and the Mukai flop $\mathcal{M}_i\dasharrow \mathcal{M}_{i+1}$ induces  birational maps
$$ M_i\dasharrow M_{i+1},\qquad \Omega_i\dasharrow \Omega_{i+1}.$$
In this section, we provide a slightly different proof of this  result that is valid for any genus $g\geq 4$. As we will see, the statement remains the same  unless $g\equiv 3 \mod 4$ and $i=\eta$, in which case the flop $\mathcal{M}_{\eta}\dasharrow \mathcal{M}_{\eta+1} $ does not induce a birational map $M_{\eta}\dasharrow M_{\eta+1}$. In fact,  its exceptional locus coincides with $M_{\eta}$. 

By Proposition \ref{mukaifloprest}, we know that  $M_i\dasharrow M_{i+1}$  and $\Omega_i\dasharrow \Omega_{i+1}$ factorizes as a series of disjoint standard  flips, divisorial contractions or extractions. In Proposition \ref{smallmod} we show that this birational map is small if $i\leq \eta$  ($i\leq \eta-1$) when  $g\not \equiv 3\mod 4$ (resp., $g\equiv 3\mod 4$). We also describe exactly the projective bundles which are flipped (see Proposition \ref{desflip}) and we obtain an explicit formula  for their ranks.

The contravariant functor $\Psi$ induces an isomorphism 
$$\mathcal{M}(v_{c,d})\overset{\sim}{\to}\mathcal{M}(v-v_{c,d}),\qquad T\mapsto \Psi(T),$$ and it induces an action  on the (possibly twisted) vector bundles $$\mathcal{V}_{c,d},\mathcal{V}^*_{c,d}\to \mathcal{M}(v_{c,d})\times \mathcal{M}(v-v_{c,d}).$$ More precisely, for a point $(T,\Psi(T'))\in \mathcal{M}(v_{c,d})\times \mathcal{M}(v-v_{c,d})$, the action of $\Psi$ sends the fiber  $\mathcal{V}_{c,d}^*|_{(T,\Psi(T'))}$ to the fiber  $ \mathcal{V}_{c,d}^*|_{(T',\Psi(T))}$ via  the linear map defined in (\ref{psiaction}). For the vector bundle $\mathcal{V}_{c,d}$, the action of $\Psi$  is the dual action via the Serre duality
$$ \Ext^1(\Psi(T'),T)\simeq \Ext^1(T,\Psi(T'))^*$$
 (see \cite[Lemma~3.24]{Flapan_2021}). Thus, when $T=T'$, by equation  (\ref{eigenrelations}), we  have 
\begin{equation}\label{eingenrelations1}
    \Ext^1(T,\Psi(T))^{\pm } =\Ann(\Ext^1(\Psi(T),T)^{\mp })
\end{equation}
For $T\in M(v_{c,d})$, we define 
$$ k_{c,d}^{\pm }\coloneqq \Dim \Ext^1(T,\Psi(T))^{\pm }=\Dim \Ext^1(\Psi(T),T)^{\pm }.$$
We will use the following vanishing lemma.
\begin{lemma}\label{vanishinglemma}
    For each $T\in \mathcal{M}(v_{c,d})$, then 
    $$ \Ext^j(T,\Psi(T))=0,\qquad j=0,2.$$
\end{lemma}
\begin{proof}
 Let $i\in I$ be the integer such that $\mu_i=\mu(c,d)$. Then $T$ fits into a destabilizing sequence (with respect to $\sigma_{\alpha_i-\varepsilon}$):
$$ T\to E\to \Psi(T),$$
 where $E$ lies in the exceptional locus of the birational map  $$\mathcal{M}_i=\mathcal{M}_{\sigma_{\alpha_i+\varepsilon}}(v)\dasharrow \mathcal{M}_{\sigma_{\alpha_i-\varepsilon}}(v)=\mathcal{M}_{i+1}.$$ 
In particular, we see that the phases (with respect to $\sigma_{\alpha_i-\varepsilon}$) satisfy the inequality 
 $$\phi(T)>\phi(\Psi(T)),$$ and so $\Hom(T,\Psi(T))=0$. An analogous argument with the birational map 
$$
\mathcal{M}_{i+1}=\mathcal{M}_{\sigma_{\alpha_i-\varepsilon}}(v)\dasharrow \mathcal{M}_{\sigma_{\alpha_i+\varepsilon}}(v)=\mathcal{M}_i
$$
shows that $\Hom(\Psi(T),T)=0$, and by the Serre duality, we obtain  the vanishing $\Ext^2(T,\Psi(T))=0$.
\end{proof}
Using the previous lemma, we obtain that 
$$ \Dim \Ext^1(T,\Psi(T))= -\chi(T,\Psi(T))=\langle v,v-v_{c,d}\rangle ,  $$
where $\langle - ,-\rangle $ denotes the Mukai pairing. Using the formulas for $v$ and $v_{c,d}$ in (\ref{v}) and (\ref{vi}) respectively, we obtain that $$\Dim \Ext^1(T,\Psi(T))=g-1-2d.$$ Therefore  $k_{c,d}^\pm$ satisfy the relation 
\begin{equation}\label{dualbundle}
     k_{c,d}^++k_{c,d}^-=\Dim\Ext^1(T,\Psi(T))= g-1-2d.
\end{equation} 
\begin{proposition}\label{vbrank}
We have
     $$k_{c,d}^{+}=\frac{c+1}{(2c+1)}(g-1-4d-(2c+1)^2)+(2c+1)c+d,$$
     and $$k_{c,d}^{-}=\frac{c}{2c+1}(g-1-4d-(2c+1^2)) +(2c+1)(1+c) +d. $$
     \begin{proof}
 The second formula follows from the first one and equation (\ref{dualbundle}). Let $T\in \mathcal{M}(v_{c,d})$.  Using   Corollary (\ref{eigenspace}) with $R=\Lambda^*[1]$ and the Serre duality, we have $$\Ext^i(T,\Psi(T))^{+}\simeq  \Ext^{i-1}(S^2T,  \Lambda^*)\simeq H^{3-i}(X,S^2T\otimes \Lambda).$$ By Lemma \ref{vanishinglemma}, we have that  $ \RHom(T,\Psi 
 (T))= \Ext^1(T,\Psi(T))[-1]$. Thus, we obtain that  $\R (X,S^2T\otimes \Lambda)=H^0(X,S^2T\otimes \Lambda)$ and so   $$k_{c,d}^+=\chi(S^2T\otimes \Lambda).$$ Now if $F$ is a complex in $D^b(X)$ with Chern classes $c_1,c_2$ and rank  $r$, then the Chern classes of $S^2F$ are given by 
$$c_1(S^2F)=c_1(r+1), \qquad c_2(S^2F)=c_1^2(r^2/2+r/2-1)+c_2(r+2). $$
Moreover, since $X$ is a K3 surface, the Mukai vector $v(F)$ is given  by
$$ v(F)=\langle r,c_1, c_1^2/2-c_2+r\rangle. $$
Thus we may use the formula for $v_{c,d}$ in (\ref{vi}) to obtain that 
$$\Rank T=2c+1,\qquad c_1(T)=-ch,\qquad c_2(T)=\frac{2c^3(g-1)+(2c+1)^2+d}{2c+1}, $$
and so 
$$\resizebox{0.91\hsize}{!}{%
$c_1(S^2T)=-ch(2c+2),\qquad c_2(S^2T)=\frac{2c+3}{2c+1}(2c^3(2c+2)(g-1)+d)+(2c+1)(2c+3).$%
}
$$
Since the Chern classes of $\Lambda$ are $c_1=h$ and $c_2=0$, we have that
\begin{equation*}
    \begin{split}
        c_1(S^2T\otimes \Lambda)&=-ch, \\
    c_2(S^2T\otimes \Lambda)&=\frac{2c(4c^4+2c^3-2c+1)(g-1)+d(2c+3)}{2c+1}+(2c+3)(2c+1).
    \end{split}
\end{equation*}
We use the Hirzebruch–Riemann–Roch Theorem to obtain the desired formula for $k_{c,d}^+$.
\end{proof}
\end{proposition}
Now we will study the restriction of the flop $\mathcal{M}_{i}\dasharrow \mathcal{M}_{i+1}$ to the connected components  of the fixed locus of $\tau$. Each flop $\mathcal{M}_{i}\dasharrow \mathcal{M}_{i+1}$ fits into the diagram (\ref{flop1}), and the contravariant functor $\Psi$, defined in (\ref{psi}), induces the antisymplectic involution $\tau$. Then, following the notation of Proposition \ref{mukaifloprest}, we have that  $$\mathcal{Z}=\mathcal{M}(v_{c,d})\times \mathcal{M}(v-v_{c,d}),\qquad \tau_\mathcal{Z}(T,\Psi(T'))=(T',\Psi(T)).$$ Note that $\Fix(\tau_\mathcal{Z},\mathcal{Z})$ is connected and it coincides with  the image of the  map
$$ \Delta: \mathcal{M}(v_{c,d})\to \mathcal{M}(v_{c,d})\times \mathcal{M}(v-v_{c,d}),\qquad T\mapsto (T,\Psi(T)). $$
Clearly, this map is an embedding,  and so we will use  the identification $\Delta(\mathcal{M}(v_{c,d}))=\mathcal{M}(v_{c,d})$. Thus, if 
$$ \mathcal{V}_{c,d}^{\pm }\subset \mathcal{V}_{c,d}|_{\mathcal{M}(v_{c,d})},\qquad  \mathcal{V}_{c,d}'^{\pm } \subset \mathcal{V}_{c,d}^*|_{\mathcal{M}(v_{c,d})},  $$
denote the  eigenbundles corresponding to the eigenspaces $$  \Ext^1(\Psi(A),A)^{\pm }\subset \Ext^1(\Psi(A),A),\qquad \Ext^1(A,\Psi(A))^{\pm }\subset \Ext^1(A,\Psi(A)),$$ 
 then $\mathbb{P}(\mathcal{V}_{c,d}^{\pm })$ and $\mathbb{P}(\mathcal{V}_{c,d}'^{\pm })$  are the connected components of $\Fix(\tau, \mathcal{P}_{c,d})$ and $\Fix(\tau, \mathcal{P}'_{c,d})$ respectively. 
 
 We have the following lemma.
 \begin{lemma}\label{lemmamukaiflop}
     Let $\mathcal{F}_i \subset \mathcal{M}_i $ be a  connected component of $\Fix(\tau,\mathcal{M}_i)$ which is not contained in $\mathcal{P}_{c,d}$ for any pair $(c,d)\in J^i$. Let $\mathcal{F}_{i+1}$ be its proper transform in $\mathcal{M}_{i+1}$. Then the flop $\mathcal{M}_{i}\dasharrow \mathcal{M}_{i+1}$ induces a birational map $\mathcal{F}_i\dasharrow \mathcal{F}_{i+1}$ which factors as a series of disjoint twisted  standard flips (not necessarily small, i.e., possibly divisoral extractions or contractions), one for each connected component of $\mathcal{P}_{c,d}\cap \mathcal{F}_i$. More precisely, each  connected component $\Gamma\subset \mathcal{P}_{c,d}\cap \mathcal{F}_i $  can be identified with a (possibly  twisted) projective bundle 
    $$  p:\Gamma\simeq\mathbb{P}(\mathcal{V}_\Gamma)\to \mathcal{M}(v_{c,d}), $$
    where $\mathcal{V}_\Gamma $ is one of the two eigenbundles $\mathcal{V}_{c,d}^\pm$. Moreover if $\mathcal{V}_\Gamma'=\Ann(\mathcal{V}_\Gamma)\subset\mathcal{V}^*|_{\mathcal{M}(v_{c,d})} $, then $\mathcal{N}_{\Gamma/\mathcal{F}_i}\simeq \mathcal{O}_p(-1)\otimes p^*\mathcal{V}_\Gamma' $, and the flip corresponding to $\Gamma$ is the twisted standard flip of $\mathcal{F}_i$  along  $\Gamma$ replacing $\Gamma\simeq\mathbb{P}(\mathcal{V}_\Gamma)$ by $\Gamma'\simeq\mathbb{P}(\mathcal{V}_\Gamma')$.
 \end{lemma}
\begin{proof}
By descent theory, every twisted projective bundle becomes untwisted after pulling back along an adequate analytic open subset of the base. Thus, we can work locally in $\overline{\mathcal{M}}_{i}$ and pulling back the diagram (\ref{flop1}) along  an analytic open neighborhood $U\subset \overline{\mathcal{M}}_{i}$ of $\mathcal{M}(v_{c,d})$  to  obtain a Mukai flop $\mathcal{M}_i|_U\dasharrow \mathcal{M}_{i+1}|_{U}$ along the untwisted projective bundle $\mathcal{P}_{c,d}|_U$. Then  Proposition \ref{mukaifloprest} provides the description of the birational map $\mathcal{F}_{i}|_U\dasharrow \mathcal{F}_{i+1}|_U$, and the proposition follows. 
\end{proof}
Now we start with the study of the restriction of each  flop to the fixed locus by $\tau$.
 \begin{proposition}\label{smallmod}
     Let  $0\leq i \leq \eta$.
    \begin{enumerate}[i)]
           \item Assume  that $i\leq \eta-1$ if $g\equiv 3 \mod 4 $. Then $\Fix(\tau,\mathcal{M}_{i+1})$ has two connected components. Moreover,  the flop $\mathcal{M}_{i}\dasharrow \mathcal{M}_{i+1}$ induces  birational maps 
           \begin{equation}\label{birrational}
               M_i\dasharrow M_{i+1},\qquad \Omega_i\dasharrow \Omega_{i+1},
           \end{equation}
           where connected $ M_i,\Omega_i$ and $M_{i+1},\Omega_{i+1}$ are the connected components of   $\Fix(\tau,\mathcal{M}_{i}) $ and $\Fix(\tau,\mathcal{M}_{i+1})$ respectively.
        \item Assume  that $i\geq 1$, and that $i\leq \eta-2$ if $g\equiv 3 \mod 4$. Then the birational maps in (\ref{birrational})  are  small, i.e., are isomorphisms in codimension one. 
    \end{enumerate}
 \end{proposition}
\begin{proof}
     By induction on $i\geq 0$,  we may assume that for all $j\leq i$,  $\Fix(\tau,\mathcal{M}_j)$ has two connected components $M_j,\Omega_j$  of dimension $g$  and there are  sequences of birational maps 
     $$  M_0\dasharrow M_1\dasharrow  \ldots \dasharrow M_i,\qquad \Omega_0\dasharrow \Omega_1\dasharrow  \ldots \dasharrow \Omega_i.$$
     Let $(c,d)$ be a pair with slope $\mu_i$. Using equation  (\ref{dualbundle}) we obtain that
    $$ \Dim \mathbb{P}(\mathcal{V}_{c,d}^\pm)=k_{c,d}^\pm-1+\Dim \mathcal{M}(v_{c,d})=k_{c,d}^\pm+1+2d=g-k_{c,d}^{\mp}.$$ Thus  parts $i)$ and $ii)$ will  follow once we show  that $k_{c,d}^\pm \geq 1$ and $k_{c,d}^\pm \geq 2$, respectively, under the given assumption. By using the slope $\mu(c,d)=\mu_i$ of equation (\ref{mucd}), the formulas of Proposition \ref{vbrank} become
    $$k_{c,d}^+=\begin{cases}
2(2c+1)(c+1)\mu(c,d)+(2c+1)c+d   & \text{ if $g$ is even, } \\
2(2c+1)(c+1)\mu(c,d)-(2c+1)+d   & \text{ if $g$ is odd,} 
\end{cases} $$
and 
$$k_{c,d}^-=\begin{cases}
2c(2c+1)\mu(c,d)+(2c+1)(c+1)+d   & \text{ if $g$ is even,} \\
2c(2c+1)\mu(c,d)+2c+1+d   & \text{ if $g$ is odd.} 
\end{cases} $$
Define 
$$ t=\begin{cases}
2(2c+1)\mu(c,d)   & \text{ if $g$ is even,} \\
(2c+1)\mu(c,d)   & \text{ if $g$ is odd.} 
\end{cases}$$
Note that the properties  $a)$ and $b)$ of the pair $(c,d)$ implies that $t$ is a positive integer. In particular $t\geq 1$ and so 
$$k_{c,d}^-\geq \begin{cases}
2c^2+3c+1+d   & \text{ if $g$ is even,} \\
4c+1+d   & \text{ if $g$ is odd.} 
\end{cases} $$
Since $d\geq -1$, we conclude that $k_{c,d}^-\geq 2$ if $c\geq 1$.  Now if $c=0$, then we have $$k_{c,d}^-=1+d,$$ and so $k_{c,d}^-\geq 1$ ($k_{c,d}^-\geq 2$) if $i\geq 0$ (resp., $i\geq 1$) since the pairs $(0,-1),(0,0)\in J$  have slope $\mu_{-1},\mu_{0}$ respectively. For $k_{c,d}^+$, we have
$$ k_{c,d}^+=\begin{cases}
(c+1)t+(2c+1)c+d   & \text{ if $g$ is even,} \\
2(c+1)t-(2c+1)+d   & \text{ if $g$ is odd.} 
\end{cases} $$
Thus if $t\geq 2$,  then  
$$k_{c,d}^+=\begin{cases}
2c^2+3c+2+d   & \text{ if $g$ is even,} \\
4c+3+d   & \text{ if $g$ is odd.} 
\end{cases} $$
where we obtain that $k_{c,d}^+\geq 2$ if $i\geq 0$ ($i>-1$ and $(0,-1)$ has slope $\mu_{-1}$). Therefore only remains to analyze the case when $t=1$. If $t=1$, then we have 
$$k_{c,d}^+=\begin{cases}
2c^2+2c+1+d  & \text{ if $g$ is even, } \\
1+d   & \text{ if $g$ is odd, } 
\end{cases}$$
which is at least 2 if $d\geq 1$. If $d=0$, then $t=1$ implies that 
$$c=\begin{cases}
\frac{g-2}{2}   & \text{ if $g$ is even, } \\
\frac{g-3}{4}  & \text{ if $g$ is odd. } 
\end{cases} $$
From this we see that $c\geq 1$ and so $k_{c,d}^+\geq 2$  if $g$ is even. We also obtain that if $g$ is odd, then $t=1$ implies that $g\equiv 3 \mod 4$ and the pair $(c,d)=(\frac{g-3}{4},0)$ has slope $\mu_{\eta-1}$, which gives $k_{c,d}^+=1$. Now, when $d=-1$, a similar argument shows that $k_{c,d}^+\geq 2$ when $g\not\equiv 3\mod 4$ and for the case $g\equiv 3\mod 4$, the pair $(c,d)=(\frac{g+1}{4},-1)\in J$ has slope  $\mu_\eta$, and it gives $k_{c,d}^+=0$.  
\end{proof}

\begin{proposition}\label{desflip}
    Let $i$ be as in Proposition \ref{smallmod} part $i)$. Then the birational map $M_i\dasharrow M_{i+1}$ ($\Omega_i\dasharrow \Omega_{i+1}$) factors as series of standard flips (not necessarily small), along  $\mathbb{P}(\mathcal{V}_{c,d}^-)$ (resp., $\mathbb{P}(\mathcal{V}_{c,d}^+)$) for each $(c,d)\in J^i$. 
\begin{proof}
    Consider an arbitrary element  $E\in M_0$.  Then $E$ fits into an exact triangle 
    $$\xymatrix{ \Psi(\mathcal{O}_X)\ar[r] & E\ar[r] & \mathcal{O}_X\ar[r]^-\alpha & \Psi(\mathcal{O}_X)[1]  }  $$
    for some $\alpha \in \Ext^1(\mathcal{O}_X,\Psi(\mathcal{O}_X))$. By Corollary \ref{eigenspace}, we have 
    \begin{equation*}
         \resizebox{1\hsize}{!}{%
         $\Ext^1(\mathcal{O}_X,\Psi(\mathcal{O}_X))^+=\Hom(S^2\mathcal{O}_X,\Lambda^*[2])=\Hom(\mathcal{O}_X,\Lambda^*[2])=\Ext^1(\mathcal{O}_X,\Psi(\mathcal{O}_X))$%
  }
    \end{equation*}
and so $\Psi$ acts trivially on this Ext group. On the other hand, $\Hom(E,\Psi(E))=\mathbb{C}$ generated by an isomorphism $t:E\to \Psi(E)$.  Thus, Lemma \ref{lemmas4} implies that $\Psi$ acts trivially on $\Hom(E,\Psi(E))$ for every  $E\in M_0$. By birationality, the same holds for all  $E\in M_{i+1}$. Let $(c,d)\in J$ be a pair with slope $\mu_i$. Since $\mathbb{P}(\mathcal{V}_{c,d}'^-)\subset \mathcal{M}_{i+1}$ parametrizes elements  which fit into extensions of the form
$$ \xymatrix{ \Psi(T)\ar[r] & E \ar[r] & T \ar[r]^-{\alpha}  &\Psi(T)[1] }$$
for some  $T\in \mathcal{M}(v_{c,d})$ and $\alpha\in \Ext^1(T,\Psi(T))^-$, we 
may apply  Lemma \ref{lemmas4} again to deduce that  $\mathbb{P}(\mathcal{V}_{c,d}'^-)\cap M_{i+1}=\emptyset$. Thus $\mathbb{P}(\mathcal{V}_{c,d}'^-)$ is contained in $\Omega_{i+1}$ since $M_{i+1},\Omega_{i+1}$ are the connected components of $\Fix(\tau,\mathcal{M}_{i+1})$. A similar application of Lemma \ref{lemmas4} implies that $\Psi$ acts as $-1$ on $\Hom(F,\Psi(F))$  for every element $F\in \Omega_{i+1}$, and so $\mathbb{P}(\mathcal{V}_{c,d}'^+)\cap \Omega_{i+1}= \emptyset$. Since  the restriction of each Mukai flop changes the  sign of the flipped eigenbundle (see  Lemma \ref{lemmamukaiflop} and equation (\ref{eingenrelations1})), we obtain that the projective bundle $\mathcal{P}(\mathcal{V}_{c,d}^-)$ ($\mathbb{P}(\mathcal{V}_{c,d}^+)$) is contained in $M_i$ (resp., $\Omega_i)$, and the statement follows from Lemma \ref{lemmamukaiflop}. 
\end{proof}
\end{proposition}
By applying this Proposition to $i=0$, we obtain the following result which was known to Macrì \cite{Macri}.
\begin{corollary}\label{blowup}
    $M_1$ is the blow-up of $M_0$ along $X$ embedded by the linear system $|\Lambda|$, and $M_0\dasharrow M_1$ is just the inverse of the blow-down morphism $M_1\to M_0$.
    \begin{proof}
    Since $(c,d)=(0,0)$ is the only pair with slope $\mu_0$, from the previous Proposition, we obtain that  $M_0\dasharrow M_1$ is a single standard flip along the projective bundle $\mathbb{P}(\mathcal{V}_0^-)$ whose base is $\mathcal{M}(v_{0,0})=\mathcal{M}(1,0,0)=X$. Moreover, by Proposition (\ref{vbrank}) we have that
$$ k_{0,0}^{+}=g-2,\qquad k_{0,0}^{-}=1.$$
  Therefore, $\mathcal{V}_{0,0}^-$ and $\mathcal{V}_{0,0}'^+$ have rank $1$ and $g-2$ respectively and the statement follows. 
    \end{proof}
\end{corollary}
In summary, if we define 
$$\nu\coloneqq\begin{cases}
\eta+1=\#I-1 & \text{ if $g\not \equiv 3\mod 4$,} \\
\eta-1=\#I-3 & \text{ if $g\equiv 3\mod 4$,} 
\end{cases}$$
we obtain a sequence of birational maps
\begin{equation}\label{sequenceflips}
    \Bl_XM_0 =M_1\dasharrow M_2\dasharrow \ldots \dasharrow M_\nu
\end{equation}
where each  birational map  $M_i\dasharrow M_{i+1}$ is small and factors as a series standard flip along the projectives bundles $\mathbb{P}(\mathcal{V}_{c,d}^-)$:
\begin{equation}\label{flippp}
    \resizebox{0.91\hsize}{!}{%
    $\xymatrixcolsep{0.9pc}\xymatrix{
   \underset{(c,d)\in J^i}{\coprod}\mathbb{P}(\mathcal{V}_{c,d}^-)\ar[ddrr] \ar@{^{(}->}[r]  &  M_i \ar@{-->}[rr] \ar[dr]  & &  M_{i+1} \ar[dl] &  \underset{(c,d)\in J^i}{\coprod}\mathbb{P}(\mathcal{V}_{c,d}'^+)\ar@{_{(}->}[l]  \ar[ddll]\\
   & &  \overline{M}_i  & & 
   \\    &  &    \underset{(c,d)\in J^i}{\coprod}\mathcal{M}(v_{c,d}) \ar@{_{(}->}[u]    & &  }$%
   }
\end{equation}
When $g\equiv 3 \mod 4$, there is an additional birational map $ M_\nu\dasharrow M_{\nu+1}$. As we will see in Section \ref{section9}, $M_{\nu+1}=\mathbb{P}^g$ and $M_\nu$ can be identified with the blow-up of $M_{\nu+1}$ along  another K3 surface. This situation differs completely from the case when $g$ is divisible by 4. In that case, there  is a divisorial contraction  $M_\nu\to \overline{M}$ into a singular Fano variety which is obtained by considering a stability condition on a Brill-Noether type wall. We will study this contraction in  Section \ref{sectiondiv}, where we provide a description its fibers. 

\section{The ample cones}\label{section8}
 By Corollary \ref{blowup}, $M_1$ is the blowup of the projective space $M_0=\mathbb{P}(H^0(X,\Lambda)^*)$ along a smooth center of codimension at least 2 (since $g\geq 4$). Thus,  $\Pic M_1$ is generated by the class of the exceptional divisor $E$  and the pullback of some hyperplane $H\subset M_0$. Following the notation of \cite[Section~5]{Thaddeus1994}, we define the line bundles $\mathcal{O}(m,n)\in \Pic M_1$:
\begin{equation}\label{linebundles}
     \mathcal{O}(m,n)\coloneqq(m+n)H-nE.
\end{equation}
By Proposition \ref{smallmod}, we have a canonical isomorphism $\Pic M_1\simeq \Pic M_i$ for $1\leq i \leq \nu$, and we denote by $\mathcal{O}_i(m,n)\in \Pic M_i$ its image under this canonical isomorphism. By abuse of notation, we allow $m,n$ to be rational numbers and so $\mathcal{O}_i(m,n)\in \Pic M_i\otimes \mathbb{Q}$. The main result of this section  describes  the ample cone of $M_{i}$ in terms of line bundles (\ref{linebundles}) and the slope $\mu_j$. 
\begin{proposition}\label{amplecone}
     For each $1\leq i < \nu$, the ample cone of $M_{i}$ is given by  
     $$\resizebox{0.91\hsize}{!}{%
     $\Amp M_{i}=\begin{cases}
\mathbb{R}_{>0}\mathcal{O}_{i}(2,g/2-1-\mu_{i-1})+\mathbb{R}_{>0}\mathcal{O}_{i}(2,g/2-1-\mu_i) & \text{ if $g$ is even,}  \\
\mathbb{R}_{>0}\mathcal{O}_{i}(2,(g-1)/2-\mu_{i-1})+\mathbb{R}_{>0}\mathcal{O}_{i}(2,(g-1)/2-\mu_i) & \text{ if $g$ is odd.}  
\end{cases}$%
}
$$

If $i=\nu$, then we have 
$$\resizebox{0.91\hsize}{!}{%
$\Amp M_{\nu} \supset \begin{cases}
\mathbb{R}_{>0}\mathcal{O}_{\nu}(2,g/2-1-\mu_{\nu-1})+\mathbb{R}_{>0}\mathcal{O}_{\nu}(2,g/2-1) & \text{ if $g$ is even,}  \\
\mathbb{R}_{>0}\mathcal{O}_{\nu}(2,(g-1)/2-\mu_{\nu-1})+\mathbb{R}_{>0}\mathcal{O}_{\nu}(2,(g-1)/2) & \text{ if $g$ is odd.}  
\end{cases}$%
}
$$
\end{proposition}
We  use following lemma which describes the ample cones of each $\mathcal{M}_i$. Recall the divisor $\tilde{a}_i=\lambda+\mu_i f\in \NS(\mathcal{M}) $.
\begin{lemma}(\cite[Theorem~3.9, Lemma~3.18]{Flapan_2021}) 
    Under the canonical isomorphism $$\NS(\mathcal{M})\simeq \NS(\mathcal{M}_{i}(v)),$$  the ample cone of $\mathcal{M}_{i}(v)$ is bounded  by $\tilde{a}_{i-1}$ and  $\tilde{a}_{i}$.  
\end{lemma}
Write $\iota:M_i\to \mathcal{M}_i$. In  order to prove  Proposition \ref{amplecone}, we will first compute the restriction of $\lambda$ and $f$ to $M_1$.
\begin{lemma}\label{lemmafl}
In $M_1$, we have 
    $$ \iota^*f=\mathcal{O}_1(0,-1),\quad  \iota^*\lambda= \begin{cases}
\mathcal{O}_1(2,g/2-1) & \text{ if $g$ is even. } \\
\mathcal{O}_1(2,(g-1)/2) & \text{ if  $g$ is odd.}  
\end{cases}$$
\end{lemma}
\begin{proof}
      For $i\in I$, let $\mathcal{E}_{i+1}$ be a quasi-universal family for the moduli space $\mathcal{M}_{i+1}$. Let $(c,d)\in J$ be a pair with slope $\mu_i$ and $\gamma$ be projective line lying in a fiber of the projective bundle $\mathcal{P}_{c,d}' \to  \mathcal{M}(v_{c,d})\times\mathcal{M}(v-v_{c,d})$.  For $w\in v^\perp_{\Alg}$  we have 
    $$\theta_{\mathcal{E}_{i+1}}(w)\cdot \gamma=-\langle w,v_{c,d}\rangle.  $$
    Indeed, for a point $(T,\Psi(T'))\in  \mathcal{M}(v_{c,d})\times \mathcal{M}(v-v_{c,d}) $ we have  $$\mathcal{P}_{c,d}'|_{(T,\Psi(T'))}=\mathbb{P}(\Ext^1(T,\Psi(T')))=:\mathbb{P}.$$ There is  a universal family $\mathcal{F}\in D^b(X\times \mathbb{P})$ which fits in a  exact triangle
    \begin{equation}\label{triangle}
        \Psi(T')\boxtimes \mathcal{O}_{\mathbb{P}}(1)\to \mathcal{F}\to T.
    \end{equation}
    Since the restriction of $\mathcal{E}_{i+1}$ to $X\times \mathbb{P}$  is equivalent to  $\mathcal{F}$, the  projection formula implies that the diagram 
    \begin{equation*}
    \xymatrix{ v^\perp \ar[rr]^-{\theta_{\mathcal{E}_{i+1}}} \ar[drr]^{\theta_\mathcal{F}}   &  & \NS(\mathcal{M}_{i+1}) \ar[d]^{\iota^*}  \\
     &    & \NS(\mathbb{P}) },
\end{equation*}
    commutes, and so
    $$ \theta_{\mathcal{E}_{i+1}}(w)\cdot \gamma =\theta_\mathcal{F}(w)\cdot \gamma=\langle w,v(\Phi_\mathcal{F} \mathcal{O}_C)\rangle.$$
    Using  the exact triangle (\ref{triangle}), we obtain 
    $$v(\Phi_\mathcal{F} \mathcal{O}_C)=2v(\Psi(T))+v(T)=2v-v_{c,d}, $$ 
    and the formula for $\theta_{\mathcal{E}_{i+1}}(w)\cdot \gamma$ follows since $w\in v^\perp$.

    In particular, for $i=1$, we can  take $\gamma$ to be a projective line in  the fiber of the exceptional divisor $E\to X$ and so:
    $$  \Deg \lambda|_{\gamma}=\begin{cases}
-\langle (2,-h,g/2-1), (1,0,0)\rangle =g/2-1 & \text{ if  $g$ is even, }  \\
-\langle (2,-h,(g-1)/2),(1,0,0)\rangle =(g-1)/2 & \text{ if  $g$ is odd,} 
\end{cases}$$ and $$ \Deg f|_{\gamma}= -\langle (0,0,-1),(1,0,0) \rangle=-1. $$
    Similarly for $i=0$, and $\gamma $  a projective line in $M_0$  we have
    $$ \Deg \lambda|_{\gamma}=\begin{cases}
-\langle (2,-h,g/2-1),(1,0,1) \rangle = g/2+1& \text{ if  $g$ is even, }  \\
-\langle (2,-h,(g-1)/2),(1,0,1)\rangle =(g+3)/2  & \text{ if  $g$ is odd, } 
\end{cases}$$ and $$ \Deg f|_{\gamma}= -\langle (0,0,-1),(1,0,1) \rangle=-1. $$
    Therefore in $M_1$, we have
    $$ \iota^*\lambda= \begin{cases}
(g/2+1)H-(g/2-1)E=\mathcal{O}_1(2,g/2-1) & \text{ if  $g$ is even,}  \\
(g+3)/2H-(g-1)/2E=\mathcal{O}_1(2,(g-1)/2)  & \text{ if  $g$ is odd,} 
\end{cases}$$ and $$\iota^*f=-H+E=\mathcal{O}_1(0,-1).$$
\end{proof}
\begin{proof}(Proposition \ref{amplecone})
    Using  Lemma \ref{lemmafl}, we see that  
    $$\iota^*\tilde{a}_i= \begin{cases}
\mathcal{O}_{i}(2,g/2-1-\mu_{j}) & \text{ if $g$ is even,} \\
\mathcal{O}_{i}(2,(g-1)/2-\mu_{j}) & \text{ if $g$ is odd.}  
\end{cases}$$
    Thus, the ample cone of $M_{i}$ contains the cone described in the statement. The fact that  $\iota^*\tilde{a}_i,\iota^*\tilde{a}_{i-1}$ are not ample for $M_i$ if $1\leq i\leq \nu-1$, follows from  the fact that  $k_{c,d}^+,k_{c',d'}^-\geq 2$ for pairs $(c,d),(c',d')$ with slopes $\mu_{i-1},\mu_{i}$ respectively. Then Proposition \ref{desflip} implies  that the divisors $\iota^*\tilde{a}_{i-1}$ and $\iota^*\tilde{a}_{i}$ induce non-trivial contraction maps.  
\end{proof}

\section{Proof of Theorem \ref{A} and \ref{B}}\label{proof}
Now we are ready prove the main results of this paper.
\setcounter{thmx}{0}
\begin{thmx}
    Let $X$ be a K3 surface of Picard number 1 and  genus $g$. If  $g\not\equiv 3 \mod 4$ then $X$ is a Fano visitor. If $g\equiv 3\mod 4$, then $X$ is a weak Fano visitor.
\end{thmx}
\begin{proof}
    Since $M_1$ is the blow up of the projective space $\mathbb{P}^g$ along the K3 surface $X$, and $M_1\dasharrow M_{i}$ is small for $1\leq i\leq \nu$, a quick computations shows that 
    $$ \omega_{M_i}^*= \mathcal{O}_i(4,g-3).$$
    Thus, from the description of the ample cone of $M_{i}$  in Proposition \ref{amplecone},  we conclude that  $M_{i}$ is Fano if and only if  $\mu_{i-1}>1/2$ and $\mu_i<1/2$ ($\mu_{i-1}>1$ and $\mu_i<1$) if $g$ is even (resp., $g$ is odd). A numerical computation shows that for $g$ even,  $\mu_i\neq 1/2$ for all $i$, and for $g$ odd, $\mu_i=1$ iff $g\equiv 3 \mod 4$ (in which case the pair $(c,d)=(0,\frac{g-3}{4})$ satisfies $\mu(c,d)=1$). Thus, since  $$\mathcal{O}(4,g-3)\in \mathbb{R}_{>0}\iota^*f+\mathbb{R}_{>0}\iota^*\lambda,$$  we obtain that   there is always  a $j\in I $ such that $M_j$ is Fano (weak Fano) if $g\not \equiv 3 \mod 4$ (resp., $g\equiv 3\mod 4$). Moreover, by  using the formulas of Proposition \ref{vbrank}, we have 
$$k_{c,d}^+-k_{c,d}^-=\begin{cases}
(2c+1)(2\mu(c,d)-1) & \text{ if $g$ is even,}  \\
 (2c+1)(2\mu(c,d)-2)& \text{ if  $g$ is odd.}  
\end{cases} $$
Therefore if $i\leq j-1$, then $k_{c,d}^+\geq k_{c,d}^-$ for every pair $(c,d)\in J$ with slope $\mu_i$. Since the standard flips are disjoint, we may apply Proposition \ref{twistedemd} and obtain a chain of fully faithful embeddings 
$$D^b(X)\hookrightarrow D^b(M_1) \hookrightarrow \ldots \hookrightarrow D^b(M_j).$$
This completes the proof of the Theorem.
\end{proof}

\begin{thmx}
Under the assumption of Theorem \ref{A}, let  $(c,d)\in J$ be a pair which  satisfies the inequality $$\frac{g-1-4d}{2(2c+1)^2}>1.$$ If $g\not\equiv 3 \mod 4$ then $\mathcal{M}(v_{c,d})$ is a Fano visitor. If $g\equiv 3\mod 4$, then $\mathcal{M}(v_{c,d})$ is a weak Fano visitor. In particular, if $d<\frac{g-3}{4}$, then $\Hilb^{d+1}(X)$ is a Fano visitor (weak Fano visitor) if $g\not\equiv 3\mod 4$  (resp., $g\equiv 3\mod 4$). 
\end{thmx}
\begin{proof}
    The condition $\frac{g-1-4d}{2(2c+1)^2}>1$ is equivalent to $k_{c,d}^+-k_{c,d}^->0$. Thus, by using Remark \ref{uuntwisteembd}, we can argue as in Theorem \ref{A}  to obtain a chain of embeddings 
    $$ D^b(\mathcal{M}(v_{c,d}))\hookrightarrow D^b(M_i)\hookrightarrow D^b(M_{i+1})\hookrightarrow \ldots \hookrightarrow D^b(M_j), $$
    where $\mu_i=\mu(c,d)$, and $M_j$ is Fano if $g\not \equiv 3 \mod 4$, and weak Fano if $g\equiv 3\mod 4$.
\end{proof}
\subsection{Examples}
Using our formulas for $k_{c,d}^\pm$ from Proposition \ref{vbrank}, we can compute the numerical data describing the flips appearing in (\ref{sequenceflips}). Here we include the cases $g=27$ and $g=28$ but the curious reader is invited to check the following code
in \cite{code} which compute it for any genus. Note that for $g=27$, each pair $(c,d)$ has an associated pair $(c',d')$ such that  $k_{c,d}^{\pm} =k_{c',d'}^{\mp}$. In the next section we will explain this symmetry and show that it is holds for every genus  $g\equiv 3\mod 4$. 
\begin{center}
    \resizebox{1\hsize}{!}{%
    $\begin{tabular} {| p {1.5 cm} | p {2 cm} | p {1.5 cm} | p {1.5 cm} |   p {1.5 cm} | p {2 cm} | p {2.3 cm} |p {2 cm} |}
    \hline
    \multicolumn{8}{ | c |  }{$g=27$}\\
    \hline
  $i$ & $(c,d)$ & $\mu_{c,d}$ & $k_{c,d}^+$ & $k_{c,d}^-$ &$k_{c,d}^+-k_{c,d}^-$ & $v_{c,d}$ & $\Dim \mathcal{M}(v_{c,d})  $   \\
     \hline
  $-1$ & $(0,-1)$ & 15 & 28 & 0 & 28 & $(1,0,1)$ & 0 \\
  0 & $(0,0)$ & 13 & 25 & 1 & 24 & $(1,0,0)$ & 2\\
 1  & $(0,1)$ & 11 & 22 & 2 & 20 & $(1,0,-1)$ & 4 \\
   2& $(0,2)$ & 9 & 19 & 3 & 16 & $(1,0,-2)$ & 6 \\
  3 & $(0,3)$& 7 & 16 & 4 & 12 & $(1,0,-3)$ & 8 \\
  4 & $(0,4)$ & 5 & 13 & 5 & 8 & $(1,0,-4)$ & 10 \\
  5 & $(0,5)$ & 3 & 10 & 6 & 4 & $(1,0,-5)$ & 12 \\
 6 & $(1,-1)$ & 5/3 & 16 & 12 & 4 & $(3,-h,9)$ & 0 \\
 7  & $(0,6)$ & 1 & 7 & 7 & 0 & $(1,0,-6)$ & 14 \\
  7 & $(1,2)$ & 1& 11 & 11 & 0 & $(3,-h,8)$ & 6 \\
  8& $(2,-1)$ & 3/5 & 12 & 16 & -4 & $(5,-2h,21)$ & 0 \\
  9 & $(1,5)$ & 1/3 & 6 & 10 & -4 & $(3,-h,7)$ & 12 \\
 10  & $(2,4)$ & 1/5 & 5 & 13 & -8 & $(5,-2h,20)$ & 10 \\
 11  & $(3,3)$& 1/7 & 4 & 16 & -12 & $(7,-3h,33)$ & 8 \\
  12 & $(4,2)$ & 1/9 & 3 & 19 & -16 & $(9,-4h,46)$ & 6 \\
  13 & $(5,1)$ & 1/11 & 2 & 22 & -20 & $(11,-5h,59)$& 4 \\
 14  & $(6,0)$ & 1/13 & 1 & 25 & -24 & $(13,-6h,72)$ & 2 \\
 15  & $(7,-1)$ & 1/15 & 0 & 28 & -28 & $(15,-7h,85)$ & 0 \\
    \hline
    \end{tabular}$%
    }
\end{center}
\vspace{0.2cm}
\begin{center}
\resizebox{1\hsize}{!}{%
     $\begin{tabular} {  | p {1.5 cm} | p {2 cm} | p {1.5 cm} | p {1.5 cm} |   p {1.5 cm} | p {2 cm} | p {2.3 cm} |p {2 cm} |}
    \hline
    \multicolumn{8}{ | c |  }{$g=28$}\\
    \hline
   $i$& $(c,d)$ & $\mu_{c,d}$ & $k_{c,d}^+$ & $k_{c,d}^-$ &$k_{c,d}^+-k_{c,d}^-$ & $v_{c,d}$ & $\Dim \mathcal{M}(v_{c,d})  $   \\
     \hline
  $-1$  & $(0,-1)$ & 15 & 29 & 0 & 29 & $(1,0,1)$ & 0 \\
  0&  $(0,0)$ & 13 & 26 & 1 & 25 & $(1,0,0)$ & 2 \\
  1&  $(0,1)$ & 11 & 23 & 2 & 21 & $(1,0,-1)$ & 4 \\
   2& $(0,2)$ & 9 & 20 & 3 & 17 & $(1,0,-2)$ & 6 \\
   3& $(0,3)$& 7 & 17 & 4 & 13 & $(1,0,-3)$ & 8 \\
   4& $(0,4)$ & 5 & 14& 5 & 9 & $(1,0,-4)$ & 10 \\
   5& $(0,5)$ & 3 & 11 & 6 & 5 & $(1,0,-5)$ & 12\\
    6& $(0,6)$ & 1 & 8 & 7 & 1 & $(1,0,-6)$ & 14 \\
   6& $(1,0)$ & 1 & 15 & 12 & 3 & $(3,-h,9)$ & 2 \\
   7& $(1,3)$ & 1/3 & 10  & 11 & -1 & $(3,-h,8)$ & 6 \\
    \hline
    \end{tabular}$%
    }
\end{center}

\section{Case $g\equiv 3 \mod 4$}\label{section9}

As observed in the previous section, the numerical table for $g=27$  has the property that for each pair $(c,d)\in J$,  there exists another pair $(c',d')\in J$  such that $k_{c,d}^{\pm}=k_{c',d'}^{\mp}$. This section focuses on studying  this symmetry  where we provide a rigorous proof that holds for any genus  $g\equiv 3 \mod 4$. In a certain sense, the pairs $(c,d)$ and $(c',d')$ are related via a derived equivalence $\Theta:D^b(X)\to D^b(\hat{X})$ where $\hat{X}$ is a dual K3 surface in the sense of Mukai. 

Let $X$ be a K3 surface with Picard rank 1 and genus $g\equiv 3\mod 4$. Then there is a diagram of birational maps 
\begin{equation}
     \resizebox{0.91\hsize}{!}{%
   $\xymatrix{  & \Bl_XM_0 = M_1 \ar[d] \ar@{-->}[r] &  M_2 \ar@{-->}[r]  &  \ldots \ar@{-->}[r] & M_{\nu} \ar@{-->}[r] & M_{\nu+1} \\
     X  \ar@{^{(}->}[r] &  \mathbb{P}^g=M_0 &   &    & &      
    }$ %
    }
\end{equation}
Consider the K3 surface $$\hat{X}=\mathcal{M}(2,-h,(g-1)/2)$$ parametrizing $h$-Gieseker stable vector bundles with Mukai vector $(2,-h,(g-1)/2)$. Since $(g-1)/2$ is co-prime with 2, by results of Mukai  \cite[Proposition~1.1,Theorem~1.2]{Mukai_1999}, we know that  $\hat{X}$ has Picard rank 1 and genus $g$. Furthermore, if $\hat{h}=c_1(\hat{\Lambda})$ denotes the  class of the ample generator $\hat{\Lambda}\in\Pic \hat{X}$, then there exists  a  normalized universal vector bundle  $\mathcal{F}\in \Coh(X\times \hat{X})$ whose first Chern class is
\begin{equation}\label{defnormalized}
    c_1(\mathcal{F})=\pi_X^*(-h)\otimes \pi_{\hat{X}}^*(\hat{h}).
\end{equation}
Since $X$ and $\hat{X}$ have the same genus, the diagram of flips associated to $\hat{X}$ has the same length:
\begin{equation}
     \resizebox{0.91\hsize}{!}{%
   $\xymatrix{  & \Bl_{\hat{X}} \hat{M_0} = \hat{M_1} \ar[d] \ar@{-->}[r] &  \hat{M}_2 \ar@{-->}[r]  &  \ldots \ar@{-->}[r] & \hat{M}_{\nu} \ar@{-->}[r] & \hat{M}_{\nu+1} \\
     \hat{X} \ar@{^{(}->}[r] &  \mathbb{P}^g=\hat{M_0} &   &    & &     
    }$ %
    }
\end{equation}
Here, each $\hat{M}_{i}$ parametrizes objects $E\in \Coh^{-1/2}\hat{X}$ with Mukai vector  $$v(E)=\hat{v}=(0,\hat{h},1-g)\in H_{\Alg}^*(\hat{X},\mathbb{Z}).$$ Following this convention, for $(c,d)\in J$, we write $\hat{v}_{c,d}$ for the Mukai vector 
$$ \hat{v}_{c,d}=(2c+1,-c\hat{h},\tfrac{(g-1)c^2-d}{2c+1}) \in H_{\Alg}^*(\hat{X},\mathbb{Z}).$$
The main result of this section is the following.
\begin{proposition}\label{involutioncor}
There is a derived equivalence $\Theta:D^b(X)\to D^b(\hat{X})$ inducing isomorphisms
    $$\Theta: M_{i}\overset{\sim}{\to} \hat{M}_{\nu+1-i}, \qquad \Theta\circ \Psi: \underset{(c,d)\in J^i}{\coprod}\mathcal{M}(v_{c,d})\overset{\sim}{\to}  \underset{(c,d)\in J^{\nu-i}}{\coprod}\mathcal{M}(\hat{v}_{c,d}),$$
   which fit into the  commutative diagram 
    \begin{equation*}
    \resizebox{1\hsize}{!}{%
    $\xymatrix{
    &  &   \underset{(c,d)\in J^i}{\coprod}\mathcal{M}(v_{c,d}) \ar[ddd]_{\Theta\circ \Psi}    & & \\
      \underset{(c,d)\in J^i}{\coprod}\mathbb{P}(\mathcal{V}_{c,d}^-) \ar[d] \ar[urr] \ar@{^{(}->}[r]  &  M_{i} \ar@{-->}[rr]   \ar[d]^{\Theta} & &  M_{i+1} \ar[d]^{\Theta} &  \underset{(c,d)\in J^i}{\coprod}\mathbb{P}(\mathcal{V}_{c,d}'^+) \ar@{_{(}->}[l] \ar[ull] \ar[d] \\
    \underset{(c,d)\in J^{\nu-i}}{\coprod}\mathbb{P}(\hat{\mathcal{V}}_{c,d}'^+) \ar@{^{(}->}[r] \ar[drr] & \hat{M}_{\nu+1-i}    & &  \hat{M}_{\nu-i} \ar@{-->}[ll] &   \underset{(c,d)\in J^{\nu-i}}{\coprod}\mathbb{P}(\hat{\mathcal{V}}_{c,d}^-) \ar@{_{(}->}[l] \ar[dll]    \\
  &  &   \underset{(c,d)\in J^{\nu-i}}{\coprod}\mathcal{M}(\hat{v}_{c,d})   & &     }$%
  }
\end{equation*}
Moreover, for each pair $(c,d)\in J^i$,  $\Theta$ restricts to an isomorphism of projectives bundles $\mathbb{P}(\mathcal{V}_{c,d}^-)\overset{\sim}{\to} \mathbb{P}(\hat{\mathcal{V}}_{c',d'}'^+)$ where $(c',d')\in J^{\nu-i}$ is the pair $(c',d')=\left(\tfrac{(g-3)/4-d-c}{2c+1},d\right)$.
\end{proposition}
The second result of this section is the following  which implies  \cite[Theorem~3.1]{cremona} when $g=7$.
\begin{proposition}\label{nonbir}
    $\hat{X}$ is not isomorphic to $X$. Let $i=(c,d)\in J$ be a pairs of integers satisfying\footnote{This is equivalent to say that $(c,d)$ has slope $\mu(c,d)=1$.}
    $$ g-1-4d=2(2c+1)^2.$$
    Then $\mathcal{M}(v_{i})\simeq\mathcal{M}(\hat{v}_{i})$. In particular for $i=(0,(g-3)/4)$, we obtain the  isomorphism $\Hilb^{(g+1)/4}X\simeq \Hilb^{(g+1)/4} \hat{X}$. 
\end{proposition}
The proof of the first statement follows the main idea used in  \cite[Section~3]{cremona}. More precisely, suppose that there is an isomorphism $f:X\simeq \hat{X}$. By using the derived equivalence $\Theta:D^b(X)\to D^b(X)$, we construct an isometry on the transcendental lattice of $\hat{X}$  which acts non-trivially on its discriminant lattice. This leads to a contradiction since for a K3 surface with  Picard rank 1, any isometry acting on the transcendental lattice must be  trivial (see \cite{Oguiso}). The second part follows from Proposition \ref{involutioncor}.

\textbf{Examples:} 
\begin{enumerate}[(a)]
    \item $g=7$. Here $\hat{X}=\mathcal{M}(2,-h,3)$. In \cite[Theorem~3.1]{cremona}, the authors proved that $\hat{X}\not\simeq X$ and so our results extends this for every K3 surface of Picard rank 1 and genus $g\equiv 3\mod 4$. Note that  $(g+1)/4=2$ and so $\Hilb^2X\simeq \Hilb^2 \hat{X}$ which  was also obtained by Yoshioka (see \cite[Example~7.2]{yoshioka}).
    \item $g=27$. Here $\hat{X}=\mathcal{M}(2,-h,13)$. Using the table of the previous Section, we have that $(0,6)$ and $(1,2)$ have slope 1 and so we obtain the isomorphisms $\Hilb^{7}X\simeq \Hilb^7 \hat{X}$ and $\mathcal{M}(3,-h,8)\simeq \mathcal{M}(3,-\hat{h},8)$. 
\end{enumerate}

Let $\Theta=\Phi_{\mathcal{F}^\vee[1]}:D^b(X)\to D^b(\hat{X})$ be the Fourier-Mukai transform with kernel $\mathcal{F}^\vee[1]$. Here $\mathcal{F}\in D^b(X\times \hat{X})$ is the normalized universal vector bundle with first Chern class given by (\ref{defnormalized}). We denote by $\vartheta:H_{\Alg}^*(X,\mathbb{Z})\to H_{\Alg}^*(\hat{X},\mathbb{Z})$ the induced isometry between the algebraic Mukai lattices of $X$ and $\hat{X}$. Let $\hat{\Psi}:D^b(\hat{X})\to D^b(\hat{X})$ be the contravariant involution given by $\hat{\Psi}(-)=(-)^\vee\otimes \hat{\Lambda}^*[1]$. 

We start with the following Lemma.
\begin{lemma}\label{latticeaction}
     $\hat{\Psi}\circ \Theta=\Theta\circ \Psi$ and $\vartheta:H_{\Alg}^*(X,\mathbb{Z})\to  H_{\Alg}^*(\hat{X},\mathbb{Z})$ satisfies the following:
     \begin{equation}
          \begin{split}
          \vartheta(0,0,-1)& =(2,-\hat{h},(g-1)/2),\\
    \vartheta(2,-h,(g-1)/2)&= (0,0,-1), \\  
    \vartheta(v)&=\hat{v}.
\end{split}
     \end{equation}
\end{lemma}
\begin{proof}
     Let $E\in D^b(X)$. By the Grothendieck-Verdier duality applied to $\pi_{\hat{X}}:X\times \hat{X}\to \hat{X}$, we have  
\begin{equation}\label{psifunctor}
    \begin{split}
        \hat{\Psi}\circ \Theta & = \SheafHom_{\hat{X}}((\pi_{\hat{X}})_*(\mathcal{F}^\vee[1]\otimes \pi_X^*E),\hat{\Lambda}^*[1]) \\
 &=(\pi_{\hat{X}})_*(\SheafHom_{X\times \hat{X}}(\mathcal{F}^\vee[1]\otimes \pi_X^*E, \pi_{\hat{X}}^* \hat{\Lambda}^*[3]))\\
 &=(\pi_{\hat{X}})_*(\mathcal{F} \otimes \pi_X^*\Lambda \otimes \pi_{\hat{X}}^*\hat{\Lambda}^*[1] \otimes \pi_{\hat{X}}^*(E^\vee \otimes \Lambda^*[1])) \\
 &= \Phi_{\tilde{\mathcal{F}}}(\Psi(E)),
    \end{split}
\end{equation}
where $\Phi_{\tilde{\mathcal{F}}}:D^b(X)\to D^b(\hat{X})$ is the Fourier-Mukai transform with kernel
$$\tilde{\mathcal{F}}=\mathcal{F}\otimes \pi_X^*(\Lambda)\otimes \pi_{\hat{X}}^*(\hat{\Lambda}^*)[1] \in D^b(X\times \hat{X}).$$
Since $\mathcal{F}$ is a rank 2 vector, then $\mathcal{F}^\vee\simeq \mathcal{F}\otimes (\det F)^*$ and by (\ref{defnormalized}), 
we obtain that $\tilde{\mathcal{F}}\simeq \mathcal{F}^\vee[1]$ and the first statement follows.

For the second part, consider a point $x\in X$. Since $\mathcal{F}$ is rank 2 vector bundle with first Chern class given by (\ref{defnormalized}), we see that $\Theta(\mathcal{O}_{x})=V[1]$ where $V$ is a vector bundle with rank 2 and $c_1(V)=\hat{h}$.  Thus $\vartheta(0,0,-1)=(2,-\hat{h},s)$ for some integer $s\in\mathbb{Z}$. Using that $\vartheta$ is an isometry, we have that $(2,-\hat{h},s)^2=\vartheta(0,0,-1)^2=0$ where we obtain that $s=(g-1)/2$. Using a similar argument with the Fourier-Mukai functor $\Theta^{-1}=\Phi_{\mathcal{F}[1]}:D^b(\hat{X})\to D^b(X)$ and a point  $\hat{x}\in \hat{X}$, we  deduce that $\vartheta(2,-h,(g-1)/2)=(0,0,-1)$. Finally, since $\vartheta$ is an isometry and that 
$$\langle v\rangle =\langle (0,0,-1), (2,-h,(g-1)/2)\rangle^\perp, \qquad\langle \hat{v}\rangle =\langle (0,0,-1), (2,-\hat{h},(g-1)/2)\rangle^\perp,$$
we see that $\vartheta(v)\in \{\pm \hat{v}\}$. During the proof of Proposition \ref{hkaction} below, we will show  that $\Theta(\Coh^{-1/2}X)=\Coh^{-1/2}\hat{X}$. Since $v$ and $\hat{v}$  lie in the  image under the Mukai vector of $\Coh^{-1/2}X$ and $\Coh^{-1/2}\hat{X}$ respectively, this forces the case $\vartheta(v)=\hat{v}$ and the Lemma follows. 
\end{proof}
Before proving Proposition \ref{involutioncor}, we first describe the action of $\Theta$ on the flop $\mathcal{M}_{i}\dasharrow \mathcal{M}_{i+1}$. 
\begin{proposition}\label{hkaction}
    The equivalence $\Theta:D^b(X)\to D^b(\hat{X})$ induces  isomorphisms 
    $$ \Theta:\mathcal{M}_{i}\overset{\sim}{\to}\mathcal{M}_{\nu+1-i},\qquad  \Theta:\underset{(c,d)\in J^{i}}{\coprod}\mathcal{M}(v_{c,d})\overset{\sim}{\to} \underset{(c,d)\in J^{\nu-i}}{\coprod} \mathcal{M}(\hat{v}-\hat{v}_{c,d}),$$
    which are compatible with the Mukai flops, i.e., they fit into the commutative diagram 
    \begin{equation*}\label{compatible}
    \resizebox{1\hsize}{!}{%
    $\xymatrix{
    &  & \underset{(c,d)\in J^{i}}{\coprod}\mathcal{M}(v_{c,d}) \times\mathcal{M}(\hat{v}-\hat{v}_{c,d}) \ar[ddd]^{\Theta\times\Theta}    & &   \\
   \underset{(c,d)\in J^{i}}{\coprod}\mathcal{P}_{c,d}\ar[d] \ar[urr] \ar@{^{(}->}[r]  &  \mathcal{M}_{i}\ar@{-->}[rr]   \ar[d]^{\Theta} & &  \mathcal{M}_{i+1} \ar[d]^{\Theta} &   \underset{(c,d)\in J^{i}}{\coprod}\mathcal{P}_{c,d}' \ar@{_{(}->}[l]  \ar[d] \ar[ull] \\
 \underset{(c,d)\in J^{\nu-i}}{\coprod} \mathcal{P}_{c,d}' \ar@{^{(}->}[r] \ar[drr] &  \mathcal{M}_{\nu+1-i}    & &  \mathcal{M}_{\nu-i} \ar@{-->}[ll] &  \underset{(c,d)\in J^{\nu-i}}{\coprod}\hat{\mathcal{P}}_{c,d} \ar@{_{(}->}[l] \ar[dll]   \\
  &  & \underset{(c,d)\in J^{\nu-i}}{\coprod}\mathcal{M}(\hat{v}-\hat{v}_{c,d})  \times \mathcal{M}(\hat{v}_{c,d})    & &     }$%
  }
\end{equation*}
Moreover, for each pair $(c,d)\in J^i$,  $\Theta$ restricts to an isomorphism of projectives bundles $\mathcal{P}_{c,d}\overset{\sim}{\to} \mathcal{P}_{c',d'}'$ where $(c',d')\in J^{\nu-i}$ is the pair $(c',d')=\left(\tfrac{(g-3)/4-d-c}{2c+1},d\right)$.
\end{proposition}
\begin{proof}
Let $\alpha\in (0,\infty)$ and $\sigma_\alpha=(Z_{\alpha,-1/2},\Coh^{-1/2}X)$ be the geometric stability condition from Section  \ref{mukaiflop}. Then $\Theta$ induces an isomorphism of (good) moduli spaces
\begin{equation}\label{sigmaequiations}
    \mathcal{M}_{\sigma_{\alpha}}(w)= \mathcal{M}_{\sigma_{\alpha}^\Theta}(\vartheta(w)),\qquad w\in H^*_{\Alg}(X,\mathbb{Z}),
\end{equation}
where $\sigma_{\alpha}^\Theta\in \Stab(\hat{X})$ is the stability condition
$$ \sigma_{\alpha}^\Theta=(Z_{\alpha,-1/2}\circ{\vartheta^{-1}},\Theta(\Coh^{-1/2} X)).$$
\underline{Claim}: $\Theta(\Coh^{-1/2} X)=\Coh^{-1/2}\hat{X}.$ 

Let $\hat{x}\in \hat{X}$ be a closed point. We show that $\mathcal{O}_{\hat{x}}\in \Theta(\Coh^{-1/2} X)$ is $\sigma_{\alpha}^\Theta$-stable with phase 1.  By construction of $\Theta$, there is a vector bundle $V\in \mathcal{M}(2,-h,(g-1)/2)$ such that $ \mathcal{O}_{\hat{x}}=\Theta(V[1]).$
Since $V$ is slope stable with slope  $$\mu(V)=\frac{(c_1(V),h)}{\Rank Vh^2}=-1/2,$$ we see that $V[1]$ is a  $\sigma_{\alpha}$-stable with phase 1. Indeed, by using the  formula of $Z_{\alpha,-1/2}$: $$Z_{\alpha,-1/2}(w)=(e^{(-1/2+i\alpha)h},w),$$ we obtain that $Z_{\alpha,-1/2}(V[1])=-\alpha^2 h^2\in \mathbb{R}_{<0}$. On the other hand, if $A\hookrightarrow V[1]$ is any stable object in $\Coh^{-1/2}X$ with phase 1, then $A$ is either a zero dimensional sheaf or of the form $F[1]$ where $F$ is a vector bundle of rank $2n$ and slope $\mu(F)=-1/2$. Thus $A$ is either zero or $A=V[1]$. 

Therefore, every skyscraper sheaf $\mathcal{O}_{\hat{x}}$ is $\sigma_{\alpha}^\Theta$-stable with phase 1 and so by the classification of hearts (see, e.g., \cite[Theorem~3.2]{huybrechtsstability}), we obtain that $\Theta(\Coh^{-1/2}X)=\Coh^{\beta}\hat{X}$ for some $\beta\in \mathbb{R}$. On the other hand, for a vector bundle $V\in \mathcal{M}(2,-\hat{h},(g-1)/2)$, there is a point $x\in X$ such that  $V[1] =\Theta(\mathcal{O}_x).$
In particular, we see that $ V[1]\in\Coh^{\beta}\hat{X}$ is $\sigma_{\alpha}^\Theta$-stable with phase 1. Therefore  by the definition of $\Coh^{\beta}\hat{X}$ (see (\ref{definitioncohb})), we conclude that  $\beta=\mu(V)=-1/2$ and so $\Theta(\Coh^{-1/2}X)=\Coh^{-1/2}\hat{X}$. This proves the claim. 

Note that here we showed that every skyscraper sheaf  is stable with phase 1. Since $\varphi_{\alpha}^\Theta$ is a good stability condition on $\Stab \hat{X}$, there is exists a unique $g\in \tilde{\Gl}^+(2,\mathbb{R})$ such that $g\cdot \sigma_{\alpha}^\Theta=\sigma_{\hat{\alpha}}$ for some $\hat{\alpha}\in (0,\infty)$ (see \cite[Proposition~10.3]{stability}). Now   compute  the value of $\hat{\alpha}$ as follows. Let $\hat{x}\in \hat{X}$ be a closed point. Then $Z_{\alpha}\circ\theta^{-1}(\mathcal{O}_{\hat{x}})=-\alpha h^2$. Since $Z_{\hat{\alpha}}(\mathcal{O}_{\hat{x}})=-1$, we see that $g$ satisfies  $g(1,0)=
(1/\alpha^2h^2,0)$ for $(1,0)\in\mathbb{R}^2=\mathbb{C}$. On the other hand, for a stable vector bundle $V\in \mathcal{M}(2,-\hat{h},(g-1)/2)$, we have that $Z_{\alpha}\circ\theta^{-1}(V[1])=Z_{\alpha}(\mathcal{O}_x)=-1$ for some point $x\in X$. Since $Z_{\hat{\alpha}}(V[1])=-\hat{\alpha}^2\hat{h}^2$ and $\hat{h}^2=h^2=2g-2$, we obtain that $\hat{\alpha}=\frac{1}{\alpha (2g-2)}$.

Since $\vartheta(v)=\hat{v}$, and the moduli spaces $\mathcal{M}_{i}$ are ordered by decreasing value of $\alpha$, using (\ref{sigmaequiations}) we conclude the isomorphism  $\mathcal{M}_{i}\simeq \hat{\mathcal{M}}_{\nu+1-i}$ for each $i\in 
 I$.

Using Lemma \ref{latticeaction} we obtain 
\begin{equation}\label{numericalvalues}
    \vartheta(v_{c,d})=\hat{v}-\hat{v}_{c',d'},\qquad i'=(c',d')=\left(\frac{(g-3)/4-d-c}{2c+1}, d\right)
\end{equation}
By  the same analysis for the Mukai vector $v_{c,d}$ instead of $v$, we obtain the isomorphism $\Theta:\mathcal{M}(v_{c,d})\simeq \mathcal{M}(\hat{v}-\hat{v}_{c',d'})$ and $\Theta:\mathcal{M}(v-v_{c,d})\simeq \mathcal{M}(\hat{v}_{c',d'})$. Finally, a more careful analysis for an stability condition $\sigma_\alpha$ in a wall with respect to $v$, reveals that these isomorphisms fit into the diagram (\ref{compatible}). The last statements follows from (\ref{numericalvalues}). Indeed, let $E\in \mathcal{P}_{c,d}$. Then $E$ fits into a non-trivial exact triangle
$$ F_1\to E\to \Psi(F_2),\qquad F_1,F_2\in \mathcal{M}(v_{c,d}).$$
By applying $\Theta$ to this triangle a using that $\hat{\Psi}\circ \Theta= \Theta\circ\Psi$, we obtain that $\Theta(E)$ fits into the triangle 
$$  \hat{\Psi}\circ \Theta \circ \Psi(F_1)\to \Theta(E)\to \Theta\circ \Psi(F_2).$$
By (\ref{numericalvalues}), we have that  $\Theta\circ \Psi(F_1),\Theta\circ \Psi(F_2)\in \mathcal{M}(\hat{v}_{c',d'})$ and so  $\Theta(E)\in \hat{\mathcal{P}}_{c',d'}'$.
\end{proof}

\begin{proof}(Proposition \ref{involutioncor})
   For $0\leq i\leq \nu+1$, we have that $M_{i}$ and $\hat{M}_{\nu+1-i}$ are the Fano-type component of $\Fix(\tau,\mathcal{M}_{i})$ and $\Fix(\hat{\tau},\hat{\mathcal{M}}_{\nu+1-i})$ respectively. Using  the relation $\hat{\Psi}\circ \Theta=\Theta\circ \Psi$, we see that $\Theta$ restricts an isomorphism between the fixed loci of $\tau$ and $\hat{\tau}$. Thus it is enough to see that $\Theta$ sends the component $M_{i}$ to $\hat{M}_{\nu+1-i}$. By birationality and the previous Proposition, it its enough to verify it when  for $i=0$. In this case, it is clear that $M_0\not\simeq \hat{\Omega}_{\nu+1-i}$\footnote{the other fixed component of  $\Fix(\hat{\tau},\hat{\mathcal{M}}_{\nu+1-i}$).} since  $\hat{\Omega}_{\nu+1-i}$ has Picard rank at least  3 and $M_0=\mathbb{P}^g$. The fact that the restriction of $\Theta$ into $M_{i}$  fits into the diagram of the Proposition follows from  (\ref{numericalvalues}) and the commutativity of the diagram:
\begin{equation*}
    \xymatrix{
 \mathcal{M}(v_{c,d}) \ar[rr]^{\hat{\Psi} \circ\Theta} \ar[d]^{\Delta}  & & \mathcal{M}(\hat{v}_{c',d'}) \ar[d]^{(\hat{\Psi}\times \hat{\Psi}) \circ \hat{\Delta}}    \\
\mathcal{M}(v_{c,d}) \times \mathcal{M}(v-v_{c,d})\ar[rr]^{\Theta\times \Theta} & & \mathcal{M}(\hat{v}-\hat{v}_{c',d'}) \times \mathcal{M}(\hat{v}_{c',d'})   }
\end{equation*}
where $\Delta$ is the map 
$$\Delta: \mathcal{M}(v_{c,d})\to \mathcal{M}(v_{c,d}) \times \mathcal{M}(v-v_{c,d}),\qquad T\mapsto (T,\Psi(T)), $$
and similarly, $\hat{\Delta}(T)=(T,\hat{\Psi}(T))$ if $T\in \mathcal{M}(\hat{v}_{c',d'})$.
\end{proof}
Now we focus on proving Proposition \ref{nonbir}. We will use the following  basic fact about lattices and symmetries.

Let $L$ be a lattice and $\varphi:L\to L$  an isometry. If $A\subset L$ is a sub-lattice preserved by $\varphi$, then so is its complement $B=A^\perp$. Thus have a well defined induced action of $\varphi$ in their discriminant lattices:
$$ \varphi:dA \to dA, \qquad  \varphi:dB\to dB,\qquad dA\coloneqq A^*/A,dB\coloneqq B^*/B.$$
If $L$ is unimodular, i.e., $L^*\simeq L$,  then via the embedding $A+B\subset L\simeq L^*$, we obtain the isomorphisms 
$$ A^*/A\simeq L^*/A+B \simeq B^*/B $$
which are clearly compatible with the action of $\varphi$.
\begin{proof}(Proof of Proposition \ref{nonbir})
    Assume that there is an isomorphism $f:X\overset{\sim}{\to} \hat{X}$. We write $L=\tilde{H}^*(\hat{X},\mathbb{Z})$, $A=\tilde{H}_{\Alg}^*(\hat{X},\mathbb{Z})$ and $B=A^\perp $. By definition, $B=T_{\hat{X}}$ is the transcendental lattice of $\hat{X}$. Let $\varphi:L\to L$ be the isometry obtained by pulling back  the Hodge-isometry $\Theta^H:\tilde{H}^*(X,\mathbb{Z})\to \tilde{H}^*(\hat{X},\mathbb{Z})=L$ induced by $\Theta$ along $f$ (see e.g \cite[Corollary~10.7]{huymukai}). Then $A$ is preserved by $\varphi$. By the previous comment, we obtain an isomorphism $A^*/A\simeq B^*/B$ which is compatible with the action of $\varphi$. 

    As a lattice, $A\simeq U\oplus \NS(\hat{X})$ where $U$ denotes the hyperbolic lattice. Since $U^*\simeq U$ and $\NS(\hat{X})^*=\NS(\hat{X})\frac{1}{\hat{h}^2}$ under the intersection pairing, we see that 
    $$ A^*/A\simeq  \left(U+\mathbb{Z}\cdot\tfrac{1}{\hat{h}^2}(0,\hat{h},0)\right)\left/ U+\mathbb{Z}\cdot(0,\hat{h},0)\right., $$
    and so $A^*/A$ is cyclic generated by $U+\tfrac{1}{\hat{h}^2}(0,\hat{h},0)$. By Lemma \ref{latticeaction}, we have that  
    $$ \varphi(U+\tfrac{1}{\hat{h}^2}(0,\hat{h},0))=U+\frac{\left(2-2g,g\hat{h},\tfrac{(1-g)(g-1)}{2}+g-1\right)}{\hat{h}^2}= U+g\cdot\tfrac{1}{\hat{h}^2}(0,\hat{h},0).$$
    That is, $\varphi$ acts by multiplication by $g$ on $A^*/A$. Since $A^*/A\simeq B^*/B$, the same holds for $ B^*/B$. In particular the action of $\varphi$ on $B$ is not trivial which contradicts the fact the only automorphism of  $T_{\hat{X}}$ is the identity (see \cite{Oguiso}). 

    For the second part, note that  if the  pair $(c,d)$ satisfies the numerical condition of the Proposition, then $(c,d)=(c',d')$ where $(c',d')$ is the  pair described at the end of the  Proposition \ref{involutioncor}. Thus we obtain   
    $\hat{\Psi}\circ \Theta:\mathcal{M}(v_{c,d})\simeq \mathcal{M}(\hat{v}_{c',d'})=\mathcal{M}(\hat{v}_{c,d})$ and the statement follows.
\end{proof}

\section{Divisoral contraction for $4\mid g$}\label{sectiondiv}
When $g$ is divisible by 4, the authors in  \cite[Lemma~3.23]{Flapan_2021} described a divisorial  contraction 
\begin{equation}\label{contraction}
    \mathcal{M}_{\eta}\to \overline{\mathcal{M}}
\end{equation}
 which is induced by a stability condition $\overline{\sigma}$ on a wall.   The involution $\tau$ on $\mathcal{M}_\eta$ descends to an involution  $\overline{\tau}:\overline{\mathcal{M}}\to \overline{\mathcal{M}}$, and by \cite[Proposition~5.12]{Flapan_2021},  the restriction of (\ref{contraction}) to the fixed locus  induces two birational morphisms  $\Omega_\eta\to \overline{\Omega}$ and $M_{\eta}\to \overline{M}$. Moreover,   $\overline{\Omega},\overline{M}$ are the irreducible components of $\Fix(\overline{\tau},\overline{\mathcal{M}})$. Denote by  $\Delta$ the divisor contracted by (\ref{contraction}) and let $\overline{\Delta}$ be its image.  By \cite[Lemma~3.24]{Flapan_2021}, each fiber of (\ref{contraction}) over a point $F\in \overline{\Delta}$  is identified with a  Grassmannian space $G_F$, which is preserved by $\tau$.  In this section, we determine  explicitly the fixed locus of  this action on $G_F$,  answering a question posed by Macr\'i.
\begin{proposition}\label{fiberdesc}
    If $F\in  \overline{\Omega}$, then $\Fix(\tau,G_F)$ is the orthogonal Grassmannian $\OGr(k,2k)$ for some $k$ (and so it is the disjoint union of its two spinor varieties).  If  $F\in  \overline{M}$, then $\Fix(\tau,G_F)$ is the Lagrangian Grassmannian $\LGr(k,2k)$.
\end{proposition}
This provides another proof of the main result of \cite[Section~5]{Flapan_2021} which states that the irreducible components   $\overline{\Omega}$  and $\overline{M}$  are disjoint. It is also worth noting that this implies that  $M_\nu\to \overline{M}$ is a divisorial contraction. Indeed, $\Delta\cap M_\eta$ is a non-empty divisor since the restriction of the line bundle $\mathcal{O}(\Delta)$ to $M_\eta$ is non-trivial. By using Proposition \ref{fiberdesc} along with a straightforward dimension computation, we conclude that  the generic fiber of  $\Delta\cap M_\eta\to \overline{\Delta}\cap \overline{M}$ is a quadric in $\mathbb{P}^4$.

To prove Proposition \ref{fiberdesc}, we begin by recalling the description of the divisorial contraction  (\ref{contraction}). Its exceptional divisor  $\Delta$ corresponds to the set of complexes $E\in \mathcal{M}_{\eta}$ such that $\Hom (A,E)\neq 0$, where $A$ denotes the unique stable spherical vector bundle with Mukai vector  $v(A)=(2,-h,g/4)$. The divisor $\Delta$ carries the natural stratification $\{\Delta_k\}_{k\geq 1}$ given by
$$ \Delta_{k}=\{E\in \mathcal{M}_\eta: \Dim \Hom(A,E)=k \}.$$
By \cite[Lemma~3.23]{Flapan_2021}, each $E\in \Delta_k$ has a $\overline{\sigma}$-destabilizing sequence  
\begin{equation}\label{triangle2}
    \xymatrix{
    A\otimes \Hom(A,E)\ar[r]^-{ev}  & E \ar[r]^a & F_k  \ar[r]^-{b} &    A\otimes \Hom(A,E)[1] },
\end{equation}
where $F_k$ is  a $\overline{\sigma}$-stable object with Mukai vector  $b_k=v-k\cdot v(A).$ For such an object $E$, we can embed the vector space $\Hom(A,E)^\vee$ into $\Hom(F_k,A[1])$   by applying the functor $\RHom(-,A)$ to (\ref{triangle2}), and using the fact that $\Hom(E,A)=0$. Conversely, for any vector subspace $W\subset \Hom(F_k,A[1])$ of dimension $k$, the  element $E$ defined (up to isomorphism) by the  exact triangle:
$$ \xymatrix{ F_k[-1] \ar[r]^-{co.ev} & A\otimes W^\vee \ar[r] & E  \ar[r] & F_k &
} $$  
lies in $\Delta_k|_{F_k}$ and satisfies $\Hom(A,E)=W^\vee$. This allows us to identify  $\Delta_k$  with a  Grassmannian bundle over $\mathcal{M}_{\overline{\sigma}}^{st}(b_k)$ whose fiber at $F_k$ is the Grassmannian space $G_{F_k}\coloneqq Gr(k,2k)$.  Each stratum $\Delta_k$ is preserved under $\Psi$. More precisely,  following the notation of (\ref{triangle2}), the $\overline{\sigma}$-destabilizing sequence for $\Psi(E)$ is given by
\begin{equation}\label{trianglepsi}
 \xymatrix{
    A\otimes \Hom(A,\Psi(E))\ar[r]^-{ev}  &  \Psi(E) \ar[r]^{a'} & \Phi(F_k)  \ar[r]^-{b'} &  A\otimes \Hom(A,\Psi(E))[1] }
\end{equation}
where $\Phi=\ST_A\circ \Psi$ (see \cite[Lemma~3.16]{Flapan_2021}). The action of  $\Psi$ on $\Delta_k$, in terms of the Grassmannian bundle description of $\Delta_k$, can be described as follows. Consider the bilinear map:
\begin{equation}\label{bil}
     \langle-,-\rangle:\Hom(F_k,A[1])\times \Hom(\Phi(F_k),A[1])\to \mathbb{C}
\end{equation}
defined by $$(\alpha,\beta)\mapsto \beta\circ \Phi(\alpha)\in \Hom(\Phi(A[1]),A[1] )\simeq \Ext^2(A,A),$$ followed by the trace morphism $\Ext^2(A,A)\to H^2(X,\mathcal{O}_X)\simeq \mathbb{C}$. We have the following result.
\begin{lemma}\label{actionspace}
   The bilinear map  (\ref{bil}) is non-degenerated. Moreover, for a  $k$-dimensional vector subspace $W\in G_{F_k} $ with associated complex $E\in \Delta_k$, the subspace 
    $$ W^\Psi:= \{\beta\in \Hom(\Phi(F_k),A[1]):\langle \alpha,\beta\rangle =0,\quad \forall \alpha\in W\}\in G_{\Phi(F_k)}$$
    is the associated vector subspace for $\Psi(E)$.
    \begin{proof}
    The first statement follows from Serre duality and the fact that $\Phi$ is an isomorphism on arrows. For  the second statement, it is enough to show that the restriction of (\ref{bil}) to $\Hom(A\otimes\Hom(A,E),A)\times \Hom(A\otimes\Hom(A,\Psi(E),A) $ is zero, where the inclusion  $\Hom(A\otimes\Hom(A,E),A)\hookrightarrow  \Hom(F_k,A[1]) $ and $\Hom(A\otimes\Hom(A,\Psi(E)),A)\hookrightarrow \Hom(\Phi(F_k),A[1])$ are obtained by applying $\RHom(-,A)$ to (\ref{triangle2}) and (\ref{trianglepsi}) respectively. Consider a pair of morphisms $$(\alpha',\beta')\in \Hom(A\otimes\Hom(A,E),A)\times \Hom(A\otimes\Hom(A,\Psi(E)),A).$$ Then its corresponding pairing is the composition $\beta'\circ b'\circ \Phi(b)\circ \Phi(\alpha')$. By the construction of the triangle  (\ref{trianglepsi})  (see the proof of \cite[Lemma~3.16]{Flapan_2021} or the proof of Lemma \ref{lemmadiagram} below), the morphism $b'$ factorizes via the composition 
    \begin{equation*}
        \xymatrixcolsep{3pc} \xymatrix{
        \Phi(F_k)\ar[r]^-{b'} \ar[d]^{\Phi(a)}  & A\otimes \Hom(A,\Psi(E))[1] \\ 
          \Phi(E) \ar[r] & A\otimes \RHom(A,\Psi(E))[1] \ar[u]
        }.
    \end{equation*}
    Thus $\beta'\circ b'\circ \Phi(b)\circ \Phi(\alpha')=(\ldots)\circ \Phi(a)\circ \Phi(b)\circ \Phi(\alpha') $ which is clearly zero. 
    \end{proof}
\end{lemma}
Now  assume that  $F_k\in \mathcal{M}_{\overline{\sigma}}^{st}(b_k)$ is fixed by $\overline{\tau}$ and choose an  isomorphism $s:\Phi(F_k)\overset{\sim}{\to}F_k$. Pulling back the space $\Hom(\Phi(F_k),A[1])$ in (\ref{bil}) along $s$ we obtain  a quadratic form 
\begin{equation}\label{bil1}
     \langle-,-\rangle:\Hom(F_k,A[1])\times \Hom(F_k,A[1])\to \mathbb{C}.
\end{equation}
By the previous Lemma, we see that Proposition \ref{fiberdesc} is equivalent to the following statement.
\begin{proposition}\label{phiaction}
    If $F_k\in  \overline{\Omega}$, then   (\ref{bil1}) symmetric. If  $F_k\in  \overline{M}$, then (\ref{bil1}) is alternating. 
\end{proposition}
Indeed, by assuming Proposition \ref{phiaction}, we can prove Proposition \ref{fiberdesc} as follows.
\begin{proof}(Proposition \ref{fiberdesc}). Let $F=F_k\in \overline{\Omega}$. By Lemma \ref{actionspace}, we know that $\Fix(\tau,G_F)$ corresponds to set of vector spaces $W\in G_k$ such that  $\langle W,W\rangle=0$ with respect to the quadratic form (\ref{bil1}). By Proposition (\ref{phiaction}), this quadratic form is symmetric because $F\in \overline{\Omega}$ and and the description of $\Fix(\tau,G_F)$ follows. The proof of the description of $\Fix(\tau,G_F)$ when $F\in \overline{M}$ is analogous. 
\end{proof}
The proof of Proposition \ref{phiaction} occupies the rest of the section. Our approach is based on studying a  natural action of  $\Phi$ on $\Hom(\Phi(F),F)$. This action is defined  using a natural isomorphism of functors $\zeta:\Id\to \Phi^2$ which satisfies certain conditions. To this end, we will briefly recall the notion of a spherical twist and its dual. We refer to \cite{Meachan2016ANO} for details and further references.

First we need some notation. Let $p,q:X\times X\to X$ be the projection to the first and second coordinates respectively.   For an object $\mathcal{P}\in D^b(X\times X)$, we denote by   $\Phi_\mathcal{P}(-)=p_*(\mathcal{P}\otimes q^*(-))$ its  associated Fourier-Mukai transform. If $Q\in D^b(X\times X)$ is another object, we write $\mathcal{P}\ast\mathcal{Q}=\pi_{13*}(\pi_{12}^*\mathcal{P}\otimes \pi_{23}^*\mathcal{Q})$. Thus the functor $\Phi_{\mathcal{Q}}\circ \Phi_{\mathcal{P}}$ is isomorphic to the Fourier-Mukai transform  $\Phi_{\mathcal{P}\ast\mathcal{Q}}$.  Similarly,   for a natural transformation of functors  $\theta:F_1\to F_2$  in $D^b(X)$ and    covariant (contravariant) equivalence  $G:D^b(X)\to D^b(X)$, we will use the notation  $\theta \ast G:F_1\circ G \to F_2\circ G$ and $G\ast \theta: G\circ F_1\to G\circ F_2$ (resp.,  $G\ast \theta: G\circ F_2\to G\circ F_1$ )  for respective induced  natural transformations.

  Denote by $\Delta:X\to X\times X$ the diagonal map. The spherical twist $\ST_A:D^b(X)\to D^b(X)$ is the Fourier-Mukai transform whose kernel $\mathcal{P}_A$ is defined (up to isomorphism) by the exact triangle  
 \begin{equation*}
   \xymatrix{
    p^*A \otimes  q^*A^\vee \ar[r]^-{\varepsilon} & \Delta_*\mathcal{O}_X \ar[r]^-{\alpha} &  \mathcal{P}_{A}  \ar[r]^-{\beta} &  p^*A \otimes  q^*A^\vee[1].
   }
 \end{equation*}
Here $\varepsilon: p^*A \otimes  q^*A^\vee \to \Delta_*\mathcal{O}_{X} $ is the morphism obtained via the adjuntion map $\Delta^* \dashv \Delta_* $ applied to the  trace map $tr:A\otimes A^\vee \to \mathcal{O}_X$. By abuse of notation, we will use the same letters to label the induced natural transformations. Thus the previous triangle induces the following sequence of natural transformations 
\begin{equation}\label{trand}
    \xymatrix{ \Phi_{p^*A \otimes  q^*A^\vee} \ar[r]^-{\varepsilon}  & \Id \ar[r]^-{\alpha} & \ST_{A} \ar[r]^-{\beta}  &  \Phi_{p^*A \otimes  q^*A^\vee[1]}.}
\end{equation}
 Similarly, we have the dual spherical twist functor $\ST_{\Psi(A)}':D^b(X)\to D^b(X)$ associated to $\Psi(A)$ which is the Fourier-Mukai transform   whose kernel $\mathcal{P}'_{\Psi(A)}\in D^b(X\times X)$ is defined (again, up to isomorphism) by the exact triangle  
\begin{equation}\label{trianglest2}
    \xymatrix{
   \mathcal{P}'_{\Psi(A)} \ar[r]^-{\delta} & \Delta_*\mathcal{O}_X \ar[r]^-{\theta} & p^*\Psi(A) \otimes  q^*\Psi(A)^\vee[2]  \ar[r]^-{\gamma}&  \mathcal{P}'_{\Psi(A)} [1].  
   }
\end{equation}
Here  $\theta:  \Delta_*\mathcal{O}_{X}\to p^*\Psi(A)\otimes q^*{\Psi(A)}^\vee[2]$  is the morphism obtained via the adjuntion map  $\Delta_* \dashv \Delta^!$  applied to the dual of the trace map $(tr)^*:\mathcal{O}_X\to \Psi(A)\otimes \Psi(A)^\vee $. As before, we obtain a sequence of natural transformation 
\begin{equation*}
  \xymatrix{
   \ST_{\Psi(A)}' \ar[r]^-{\delta} & \Id \ar[r]^-{\theta} & \Phi_{p^*\Psi(A) \otimes  q^*\Psi(A)^\vee[2]}  \ar[r]^-{\gamma}&  \ST_{\Psi(A)}' [1]  
  }.
\end{equation*}
Write $\Phi'=\Psi\circ \ST_{\Psi(A)}'$. We have the following Lemma
\begin{lemma}
    There are  isomorphisms of functors $\varphi:\Phi\to \Phi',\phi:\ST_{\Psi(A)}'\circ \ST_A\to \Id $ which fit into the commutative diagrams of functors
\begin{equation}\label{diagrams}
    \xymatrixcolsep{4pc}\xymatrix{ 
    \Psi \ar[r]^{ \alpha\ast\Psi} \ar[dr]_-{\Psi\ast \delta } &   \Phi \ar[d]^{\varphi}  \\
     &   \Phi'
    } \qquad \text{and} \qquad \xymatrixcolsep{4pc}\xymatrix{   \ST_{\Psi(A)}'\circ \ST_A \ar[r]^-{\phi}   \ar[dr]_-{\delta\ast \ST_A } &   \Id \ar[d]^{\alpha}  \\  & \ST_A  
    }.
\end{equation}
\end{lemma}
\begin{proof}
    Consider the functor $$\tilde{\Psi}:D^b(X\times X)\to D^b(X\times X),\qquad \tilde{\Psi}(-)=\SheafHom(-,p^*\Lambda^*[1]\otimes q^*(\Lambda^*[1])^\vee[2]).$$ By the Grothendieck-Verdier duality applied morphism $\Delta$, we obtain an isomorphisms $\Tilde{\Psi}(\Delta_*\mathcal{O}_X)\simeq \Delta_* \mathcal{O}_X$. Moreover, we have that $\tilde{\Psi}(p^*A\otimes q^*A^\vee)\simeq p^*\Psi(A)\otimes q^*\Psi(A)^\vee[2]$ and   under these isomorphisms, $\tilde{\Psi}(\varepsilon)$ corresponds via the adjuntion map $\Delta_* \dashv \Delta^!$  to the dual of the trace map $(tr)^*:\mathcal{O}_X\to \Psi(A)\otimes \Psi(A)^\vee$. Thus we obtain a commutative diagram 
\begin{equation}\label{specialmorph}
    \xymatrixcolsep{4pc}\xymatrix{
    \mathcal{P}_{\Psi(A)}'\ar[r]^-{\delta}  \ar@{-->}[d]^{t} & \Delta_*\mathcal{O}_X \ar[r]^-{\theta} \ar[d] &  p^*\Psi(A)\otimes \Psi(A)^\vee[2] \ar[d] \\
    \tilde{\Psi}(\mathcal{P}_A) \ar[r]^-{\tilde{\Psi}(\alpha_R)} &\tilde{\Psi}(\Delta_*\mathcal{O}_X) \ar[r]^-{\tilde{\Psi}(\varepsilon)} & \tilde{\Psi}(p^*A\otimes q^*A^\vee)
    },
\end{equation}
and so we can find an isomorphism $t:\mathcal{P}_{\Psi(A)}'\to  \tilde{\Psi}(\mathcal{P}_A)$ which complete the diagram. On the other hand, for any $\mathcal{E}\in D^b(X\times X)$, by  functoriality of the Grothendieck-Verdier duality applied  to the projection $p:X\times X\to X$, we obtain an  isomorphism of functors  $ \Phi_{\tilde{\Psi}(\mathcal{E})}\circ \Psi^2\to \Psi\circ \Phi_\mathcal{E}\circ \Psi $ (see the computation made in (\ref{psifunctor})). In particular, for $\mathcal{E}=\mathcal{P}_A$ and $\mathcal{E}=\Delta_*\mathcal{O}_X$ we obtain  natural isomorphism  $\Phi_{\tilde{\Psi}(\mathcal{P}_A)}\circ \Psi^2\to \Psi\circ \ST_A\circ \Psi $ and $ \Phi_{\Tilde{\Psi}(\Delta_*\mathcal{O}_X)}\circ \Psi^2\to \Psi\circ \Phi_{\Delta_*\mathcal{O}_X}\circ \Psi=\Psi^2 $ which fit into the commutative diagram (again, by the functoriality of the Grothendieck-Verdier duality)
\begin{equation*}
    \xymatrix{
    \Phi_{\tilde{\Psi}(\mathcal{P}_A)}\circ \Psi^2 \ar[r] \ar[d]^{\tilde{\Psi}(\alpha)\ast \Psi^2} & \Psi\circ \ST_A\circ \Psi  \ar[d]^-{\Psi\ast\alpha\ast \Psi}\\
 \Phi_{\Tilde{\Psi}(\Delta_*\mathcal{O}_X)}\circ \Psi^2 \ar[r] & \Psi\circ \Phi_{\Delta_*\mathcal{O}_X}\circ \Psi=\Psi^2
    }.
\end{equation*}
Now consider the diagram 
\begin{equation*}
    \resizebox{\displaywidth}{!}{%
    \xymatrixcolsep{5pc}\xymatrix{
    \Psi^2\circ \ST_{\Psi(A)} \ar[r]^-{\eta^{-1}\ast\ST_{\Psi(A)}} \ar[d]_-{\Psi^2\ast\delta}& \ST_{\Psi(A)} \ar[r]^{t} \ar[d]_-{\delta}& \Phi_{\tilde{\Psi}(\mathcal{P}_A)} \ar[rr]^-{\eta\ast \Psi^2} \ar[d]_-{\tilde{\Psi}(\alpha)} & &\Phi_{\tilde{\Psi}(\mathcal{P}_A)}\circ \Psi^2 \ar[r] \ar[d]^{\tilde{\Psi}(\alpha)\ast \Psi^2} & \Psi\circ \ST_A\circ \Psi  \ar[d]^-{\Psi\ast\alpha\ast \Psi}\\
 \Psi^2 \ar[r]^-{\eta^{-1}} & \Id   \ar[r] & \Phi_{\tilde{\Psi}(\Delta_*\mathcal{O}_X)} \ar[rr]^-{\Phi_{\tilde{\Psi}(\Delta_*\mathcal{O}_X)}\ast\eta} & & \Phi_{\Tilde{\Psi}(\Delta_*\mathcal{O}_X)}\circ \Psi^2 \ar[r] & \Psi\circ \Phi_{\Delta_*\mathcal{O}_X}\circ \Psi=\Psi^2,
 }}
\end{equation*}
where $\eta=\eta^{\Lambda^*[1]}:\Id\to \Psi^2$ is the natural transformation defined in Section \ref{section4}.   The first and third square  (from left to right)  are clearly commutative. The second square commutes since it is the diagram of functors induced by the left square in (\ref{specialmorph}). Thus, since the last square is also commutative, the entire diagram commutes. Moreover, by following the construction of the natural transformations, it is easy to see that the composition 
  $$  \resizebox{1\hsize}{!}{%
  \xymatrixcolsep{4pc}\xymatrix{
  \Id   \ar[r] & \Phi_{\tilde{\Psi}(\Delta_*\mathcal{O}_X)} \ar[rr]^-{\Phi_{\tilde{\Psi}(\Delta_*\mathcal{O}_X)}\ast\eta} & & \Phi_{\Tilde{\Psi}(\Delta_*\mathcal{O}_X)}\circ \Psi^2 \ar[r] & \Psi\circ \Phi_{\Delta_*\mathcal{O}_X}\circ \Psi=\Psi^2
    }%
    }
  $$
coincides with  $\eta:\Id\to \Psi^2$. Therefore we obtain a natural isomorphism of functors $\varphi':\Psi^2\circ \ST_{\Psi(A)}\to \Psi\circ \ST_A\circ \Psi $ which fits into the commutative diagram
 \begin{equation*}
      \xymatrix{
    \Psi^2\circ \ST_{\Psi(A)} \ar[r]^-{\varphi'} \ar[d]_-{\Psi^2\ast\delta} & \Psi\circ \ST_A\circ \Psi  \ar[d]^-{\Psi\ast\alpha\ast \Psi}\\
 \Psi^2 \ar[r]^-{\Id} & \Psi^2.
 }
 \end{equation*}
 Then $\varphi\coloneqq \Psi^{-1}\ast \varphi': \Psi\circ \ST_{\Psi(A)}\to \ST_A\circ \Psi $ satisfies the required conditions. 
 
 For the construction of the second  natural isomorphism $\phi:\ST_{\Psi(A)}'\circ\ST_A\to \Id$ fitting in (\ref{diagrams}), consider the exact triangle defining the dual twist functor $\ST_A'$ associate to $A$:
\begin{equation}\label{triangles3}
   \xymatrix{
    \mathcal{P}'_{A} \ar[r]^-{\tilde{\delta}} & \Delta_*\mathcal{O}_X \ar[r]^-{\tilde{\theta}} & p^*A \otimes  q^*A^\vee[2]  \ar[r]^-{\tilde{\gamma}}&  \mathcal{P}'_{A} [1].
   } 
\end{equation}
During the  proof of \cite[Theorem~2.3]{Meachan2016ANO}, the author constructed a natural isomorphism of functors $\Id\to \ST_A\circ \ST_A'$ whose composition with  $\tilde{\delta}:\ST_A'\to \Id $ is the natural transformation $\alpha\ast \ST_A':\ST_A'\to \ST_A\circ \ST_A'$. We can apply the same ideas in order to obtain a   natural isomorphism  $\ST_A'\circ \ST_A\to \Id $ whose composition with  $\alpha:\Id\to \ST_A$ is the map $\tilde{\delta}\ast \ST_A:\ST_A'\circ \ST_A\to  \ST_A$. Now since $\Psi(A)\overset{\sim}{\to} A[1]$, we obtain an isomorphism $\lambda:\mathcal{P}_{\Psi(A)}'\to \mathcal{P}_A'$ fitting into the  isomorphism of triangles  (\ref{trianglest2}) and (\ref{triangles3}) induced by $\Psi(A)\overset{\sim}{\to} A[1]$. The natural transformation $\phi:\ST_{\Psi(A)}'\circ \ST_A\to \Id$ given by the composition  $\ST_{\Psi(A)}'\circ \ST_A\overset{\lambda \ast\ST_A}{\to} \ST_{A}'\circ \ST_A \to \Id $ satisfies the required conditions. 
\end{proof}
Using the natural transformation  $\varphi$ and $\phi$  we construct the   isomorphism $\zeta:\Id\to \Phi^2$ as follows: we set $\zeta:\Id\to \Phi^2$ to be the natural transformation that completes the commutative square
\begin{equation}\label{zetadef}
     \xymatrixcolsep{5pc}\xymatrix{
    \Id \ar[r]^{\zeta} \ar[d]_-{\eta} & \Phi^2 \\
     \Psi^2  \ar[r]^-{\Psi\ast  \phi  \ast \Psi }   & \Phi'\circ \Phi \ar[u]_-{ \varphi^{-1}\ast \Phi}.
    }
\end{equation}
For each $F\in D^b(X)$, we define the  linear map 
$$ \Phi:\Hom(\Phi(F),F)\to \Hom(\Phi(F),F),\qquad s\mapsto s^\Phi=\zeta_F^{-1}\circ \Phi(s),$$
and similarly, for a morphism  $ s\in\Hom(F,\Phi(F))$, we define $s^\Phi=\Phi(s)\circ \zeta_F$. By using (\ref{zetadef}), it is straightforward to check that the following diagram commutes 
\begin{equation}\label{phipsiaction}
    \xymatrix{
    \Hom(\Phi(F),F) \ar[r]^{\Phi} \ar[d]^{\ST_{\Psi(A)}'(-)} & \Hom(\Phi(F),F) \ar[dd]^{\circ \varphi^{-1}_F } \\
    \Hom(\ST_{\Psi(A)}'\circ \Phi(F),\ST_{\Psi(A)}'(F)) \ar[d]^{\circ \phi_{\Psi(F)}^{-1}  }& \\
    \Hom(\Psi(F), \ST_{\Psi(A)}'(F)) \ar[r]^\Psi & \Hom(\Phi'(F),F).
    }
\end{equation}
Here the bottom arrow  denotes the  map which sends a morphism  $t\in \Hom(\Psi(F),F)$ to  $t^\Psi= \eta_{F}^{-1}\circ\Psi(t)$ (see Section \ref{section4}). 

To prove Proposition \ref{phiaction}, we first get  rid of the dependency of (\ref{bil1}) in the isomorphism $s:\Phi(F_k)\to F_k$ by  considering the multi-linear map 
\begin{equation}\label{bil2}
  \resizebox{0.91\hsize}{!}{%
  \xymatrixrowsep{0.2pc} \xymatrix{ \Hom(F_k,A[1])\times \Hom (\Phi(F_k),F_k) \times \Hom(\Phi(F_k),A[1]) \ar[r] &  \Hom(\Phi(A[1]),A[1])\simeq \mathbb{C} \\
   (f,s,g) \ar@{|->}[r] &  g\circ s \circ \Phi(f)
   } %
   }
\end{equation}
The following Lemma is straightforward.
\begin{lemma}\label{phiactionst}
    Consider $f,g$ and $s$ as in (\ref{bil2}). Then 
    $$ (f,s,g)^\Phi=(f,s^\Phi,g).$$
\end{lemma}
Using this Lemma, we reduce  the proof of  Proposition \ref{phiaction} to the  study the action of $\Phi$ on the vector space $\Hom(\Phi(F),F)$  for the cases  $F=A[1]$ and $F=F_k$. We start with the case $F=A[1]$. 
\begin{lemma}\label{actiononA}
    $\Phi$ acts as $-1$ on $\Hom(\Phi(A[1]),A[1])$. 
\end{lemma}
\begin{proof}
    For the action of $\Psi$ on $\Hom(\Psi(F),G)$ using $\zeta_F^{-1}$, we can apply the same argument used in the proof Lemma \ref{lemma4.1} to obtain the following commutative diagram
    \begin{equation}
        \xymatrix{ 
        \Hom(\Psi(F),G)\ar[r]^-{\Psi} \ar[d] & \Hom(\Psi(G),F) \ar[d] \\
        \Hom(\Lambda^*[1],F\otimes G) \ar[r]^{\mu\circ } & \Hom(\Lambda^*[1],G\otimes F).
        }
    \end{equation}
      By using this  diagram   with $F=F_k$ and $G=\ST_{\Psi(A)}'$, and the diagram  (\ref{phipsiaction}), we obtain the commutative square
\begin{equation*}
    \xymatrix{
    \Hom(\Phi(A[1]),A[1]) \ar[r]^{\Phi} \ar[d] & \Hom(\Phi(A[1]),A[1]) \ar[d]\\
      \Hom( \Lambda^*[1], A[1]\otimes \ST_{\Psi(A)}'(A[1])) \ar[r]^{\mu\circ } & \Hom(\Lambda^*[1], \ST_{\Psi(A)}'(A[1])\otimes A[1] ).
    }
\end{equation*}
After choosing an isomorphism  $\ST_{\Psi(A)}'(A[1])\simeq A[2]$, we see that  $\mu:A[1]\otimes A[2]\to A[2]\otimes A[1]$ is given by $\mu(a\otimes b)=b\otimes a$  (for $a\in (A[1])^{-1}=A$ and $b \in A[2]^{-2}=A$). Thus, $\Phi$  acts as multiplication by $-1$ if every map in $ \Hom( \Lambda^*[1], A[1]\otimes A[2])$ factorizes through $(\Lambda^2 A)[3]\to A[1]\otimes A[2]$. Since $\Hom(\Lambda^*[1],(\Lambda^2 A)[3])\simeq \dim H^2(X,\mathcal{O}_X)$ is 1-dimensional, and  $\Hom(\Phi(A[1]),A[1])=$ is also 1-dimensional, the Lemma follows.
\end{proof}
For the case $F=F_k$, first note that for any isomorphism $s\in \Hom(\Phi(F_k),F_k)$, we have  $s^{\Phi}=\pm s$ iff $(s^{-1})^{\Phi}=\pm s^{-1}$, so we reduce to study the action of $\Phi$ on $\Hom(F_k,\Phi(F_k))$. Consider an object $E\in \Delta_k|_{F_k}$ which is fixed by $\tau$ and choose an isomorphism $t:E\to \Psi(E)$. Then $t$ induces an isomorphism of exact triangles 
\begin{equation}\label{psiphiaction}
    \xymatrix{
  A\otimes \Hom(A,E)  \ar[r]^-{ev} \ar[d] &   E  \ar[d]^t \ar[r]^{a} & F_k \ar[d]^s \ar[r]^-{b} &  A\otimes \Hom(A,E)[1]  \ar[d]  \\
  A\otimes \Hom(A,\Psi(E)) \ar[r]^-{ev} & \Psi(E)  \ar[r]^{a'}  & \Phi(F_k) \ar[r]^-{b'}&   A\otimes \Hom(A,\Psi(E))[1]}
\end{equation}
for some isomorphism $s:F_k\to \Phi(F_k)$. The action of $\Psi$ on $t$  and the action of $\Phi$ on $s$  are related via the following result:
\begin{lemma}\label{lemmadiagram}
The  diagram 
\begin{equation*}
    \xymatrix{
  E \ar[d]^{t^\Psi} \ar[r]^{a} & F_k \ar[d]^{s^\Phi}  \\
   \Psi(E)  \ar[r]^{a'}  & \Phi(F_k) }
\end{equation*}
commutes.
\begin{proof}
Write $V=\Hom(A,E), V'=\Hom(A,\Psi(E))$ and 
\begin{equation*}
    \begin{split}
        K=\RHom(A, \Psi(A\otimes V[1])), \qquad & \quad  K'=\RHom(A, \Psi(A\otimes V'[1])), \\
        U=\RHom(A,\Psi(F_k)), \quad  &  \qquad U'=\RHom(A,\Psi\circ\Phi(F_k)),\\
        Q=\RHom(A, \Psi(E)),\quad  & \qquad Q'=\RHom(A, \Psi^2(E)).
    \end{split}
\end{equation*}
We have a commutative diagram 
\begin{equation*}
     \resizebox{1\hsize}{!}{%
    \xymatrix{  
  &  A\otimes K'\ar@{-->}[dd] \ar[rr] \ar[dl] &  & A\otimes U' \ar[rr] \ar@{-->}[dd] \ar[dl]  &  & A\otimes Q' \ar[dd] \ar[dl] \\
 A\otimes K \ar[rr] \ar[dd] & & A\otimes U  \ar[rr] \ar[dd] & &  A\otimes Q \ar[dd]  & \\
 & \Psi(A\otimes V'[1] ) \ar@{-->}[rr] \ar@{-->}[dd] \ar[dl] & & \Psi\circ \Phi(F_k) \ar@{-->}[rr]^>>>>>>>>>{\Psi(a')} \ar[dl]  \ar@{-->}[dd]^<<<<<<{\alpha_{\Psi\circ\Phi(F_k)}} & & \Psi^2(E) \ar[dd] \ar[dl]^{\Psi(s)}  \\
 \Psi(A\otimes V[1]) \ar[rr] \ar[dd] & & \Psi(F_k) \ar[rr]^>>>>>>>>>{\Psi(a)} \ar[dd]^<<<<<<{\alpha_{\Psi(F_k)}} & &  \Psi(E) \ar[dd] & \\
& \Phi(A\otimes V'[1]) \ar[dl] \ar@{-->}[rr] & & \Phi^2(F_k)\ar@{-->}[rr] \ar[dl]^{\Phi(t)} & & \Phi\circ \Psi(E) \ar[dl] \\
 \Phi(A\otimes V[1]) \ar[rr]  & & \Phi(F_k) \ar[rr] & & \Phi(E)&}%
 }
\end{equation*} 
 By  applying the octahedral axiom to the compositions 
$$ A\otimes K^{0}\to A\otimes U\to  \Psi(F_k),\quad \text{ and }\quad  A\otimes K'^{0}\to A\otimes U'\to  \Psi\circ\Phi(F_k)$$
with $K^0=\Hom(A,\Psi(A\otimes V[1]))$ and $K'^0=\Hom(A,\Psi(A\otimes V'[1]))$, we obtain the morphism of exact triangles
\begin{equation}\label{distriangles}
 \resizebox{0.91\hsize}{!}{%
\xymatrix{
A\otimes \Hom(A,\Psi^2(E))\ar[r]^-{ev} \ar[d] & \Psi^2(E) \ar[d]^{\Psi(t)} \ar[r]^{a''} & \Phi^2(F_k) \ar[d]^{\Phi(s)} \ar[r] & A\otimes \Hom(A,\Psi^2(E))[1] \ar[d] \\
A\otimes \Hom(A,\Psi(E))\ar[r]^-{ev} & \Psi(E) \ar[r]^{a'} & \Phi(F_k) \ar[r] & A\otimes \Hom(A,\Psi(E)), }%
}
\end{equation}
where the morphisms $a'$ and $a''$ fit into the  commutative diagrams 
\begin{equation}\label{triangles1}
    \resizebox{0.91\hsize}{!}{%
    \xymatrixcolsep{4pc}\xymatrix{
    \Psi(F_k) \ar[r]^{\Psi(a)} \ar[dr]_{\alpha_{\Psi(F_k)}} & \Psi(E) \ar[d]^{a'}  & & \Psi\circ \Phi(F_k) \ar[r]^{\Psi(a')} \ar[dr]_{\alpha_{\Psi\circ\Phi(F_k)}} & \Psi^2(E)  \ar[d]^{a''} \\
     &  \Phi(F_k) & & & \Phi^2(F_k).
    }%
    }
\end{equation}
Now let us consider the following diagram 
\begin{equation}\label{bigsquare}
\xymatrixrowsep{4pc}\xymatrixcolsep{5pc}\xymatrix{
 \Psi\circ \Phi(F_k) \ar[dr]_{\Psi(a')} \ar@/^2pc/[drr]^{\Psi(\alpha_{\Psi(F_k)})} \ar@/_2pc/[ddr]_{\alpha_{\Psi\circ\Phi(F_k)}} \ar@{}[ddr] | {III} &  &  \\
     & \Psi^2(E) \ar@{}[u] | {II} \ar[r]^{\Psi^2(a)} \ar[d]^{a''} \ar@{}[dr] | {I}  & \Psi^2(F_k) \ar[d]^{\Psi(\phi_{\Psi(F_k)}) }  \\
     & \Phi^2(F_k)  \ar[r]^-{\varphi_{\Phi(F_k)}}   &   \Phi'\circ \Phi(F_k) 
    }
\end{equation}
\underline{Claim} The square $I$ commutes.\\
The triangles  $II$ and $III$ are commutative by (\ref{triangles1}). The exterior square also commutes. Indeed,  if we consider the exact triangle 
\begin{equation*}
    \xymatrixcolsep{3.5pc}\xymatrix{ 
    \ST_{\Psi(A)}'\circ \Phi(F_k)\ar[r]^-{\delta_{\Phi(F_K)}}& \Phi(F_k)\ar[r]^-{co.ev} & \Psi(A)\otimes\RHom(\Phi(F_k), \Psi(A))^\vee.
    }
\end{equation*}
Then the map $\Psi(\delta_{\Phi(F_K)})$ divides the exterior square into two two triangles 
\begin{equation*}
\xymatrixrowsep{4pc}\xymatrixcolsep{5pc}\xymatrix{
 \Psi\circ \Phi(F_k) \ar[ddrr]^{\Psi(q)} \ar@/^2pc/[drr]^{\Psi(\alpha_{\Psi(F_k)})} \ar@/_2pc/[ddr]_{\alpha_{\Psi\circ\Phi(F_k)}} &  &  \\
     & & \Psi^2(F_k)  \ar@{}[l] | {V} \ar[d]^-{\Psi(\phi_{\Psi(F_k)})}  \\
     & \Phi^2(F_k) \ar[r]^-{\varphi_{\Phi(F_k)}} \ar@{}[u] | {IV}   &    \Phi'\circ \Phi(F_k). 
    }
\end{equation*}
The triangles $IV$ and $V$ commute since they are obtained by applying the diagram of functors in  (\ref{diagrams})  to  $\Phi(F_K)$ and $\Psi(F_k)$ respectively. Thus from (\ref{bigsquare}), we obtain that $$ (\Psi(\phi_{\Psi(F_k)})\circ\Psi^2(a))\circ\Psi(a')=   (\varphi_{\Phi(F_k)}\circ a'')\circ \Psi(a').$$
Moreover,  since $\Hom(\Psi(A\otimes \Hom(A,\Psi(E))[1]),\Phi^2(F_k))=0 $, then $\Psi(\phi_{\Psi(F_k)})\circ\Psi^2(a)= \varphi_{\Phi(F_k)}\circ a'' $ and the square $I$ commutes. Finally, from the construction of the natural transformation $\zeta:\Id\to \Phi^2$ (see (\ref{zetadef})), the following diagram commutes
\begin{equation*}
    \xymatrixcolsep{4pc}\xymatrix{E \ar[r]^a  \ar[d]^{\eta_E} & F_k \ar[d]^{\eta_F}  \ar@/^2pc/[dd]^{\zeta_F} \\ 
    \Psi^2(E) \ar[r]^{\Psi^2(a)} \ar[dd]^{\Psi(s)} \ar[dr]^{a''} & \Psi^2(F_k) \ar[d] \\ 
    & \Phi^2(F_k) \ar[d]^{\Phi(s)} \\
    \Psi(E) \ar[r]^{a'} &  \Phi(F_k) 
    }
\end{equation*}
and the Lemma follows.
\end{proof}
\end{lemma}
Now we are ready to prove Proposition (\ref{phiaction}).
\begin{proof}(Proposition \ref{phiaction} )
Consider an element  $F_k\in \Fix(\overline{\tau},\overline{\Delta})$ and  let $E\in \Delta|_{F_k}$ be  a $\Psi$-invariant element. In the notation of (\ref{psiphiaction}), the morphisms $a,a''$ are both non-zero. Thus  Lemma  \ref{lemmadiagram} implies that $t^\Psi=\pm t$ iff $s^\Phi=\pm s$. Moreover, from the proof of Proposition \ref{desflip}, we know that 
 $$ t^\Psi=\begin{cases} t \text{ if $E\in M_\eta$}, \\
-t \text{ if $E\in \Omega_\eta$}.
\end{cases}$$
Then the statement follows from Lemma \ref{phiactionst}  and \ref{actiononA}.
\end{proof}
\bibliographystyle{plain}
\bibliography{references}
\end{document}